\documentclass[onefignum,onetabnum]{siamart190516}
\usepackage{amsmath, amsfonts,amssymb,color}
\usepackage[english]{babel}
\usepackage{dsfont}
\usepackage{mathabx}
\usepackage{enumitem}
\usepackage{graphicx,epsfig}
\usepackage{float}
\usepackage{capt-of}
\usepackage{tikz}

\newcommand{\R}{{\mathbb R}}
\newcommand{\N}{{\mathbb N}}

\newcommand{\Z}{{\mathbb Z}}
\newcommand{\T}{{\mathbb T}}

\newcommand{\fper}{f_{{\rm per}}}
\newcommand{\uper}{u_{{\rm per}}}
\newcommand{\ellper}{\ell_{{\rm per}}}
\newcommand{\Hper}{H_{{\rm per}}}
\newcommand{\chiperp}{\chi_{{\rm per},p}}

\newcommand{\upereps}{u_{{\rm per},\varepsilon}}
\newcommand{\gpereps}{g_{{\rm per},\varepsilon}}
\newcommand{\hpereps}{h_{{\rm per},\varepsilon}}
\newcommand{\gzeroeps} {g_{0,\varepsilon}}
\newcommand{\hzeroeps} {h_{0,\varepsilon}}

\let\ds=\displaystyle
 \definecolor{mypurple}{RGB}{140,0,255}%{153,0,255}
\definecolor{myred}{rgb}{255,0,0}
\definecolor{mydarkturquoise}{RGB}{0,206,209}
\definecolor{mydeeppink}{RGB}{255,20,147}
\definecolor{darkblue}{RGB}{0,0,140}
\definecolor{blue2}{RGB}{0,0,0}
\definecolor{middleblue}{RGB}{0,0,71}
\definecolor{light-gray}{gray}{0.9}
\definecolor{ProcessBlue}{cmyk}{1,0,0,0.40}
\definecolor{Black}{cmyk}{0,0,0,1}
\definecolor{Red}{cmyk}{0,1,1,0.2}
\definecolor{Green}{cmyk}{0.9,0,1,0}
\definecolor{Orange}{cmyk}{0,0.61,0.87,0.5}
\definecolor{Fuchsia}{cmyk}{0.47,0.91,0,0.06}
\definecolor{PineGreen}{cmyk}{0.92,0,0.59,0.30}

\numberwithin{equation}{section}
 \def\theequation{\thesection.\arabic{equation}}
 \newtheorem{remark}{Remark}[section]

\DeclareMathOperator*{\argmin}{arg\,min}

\begin{document}

\title{Homogenization of some periodic Hamilton-Jacobi equations with defects}

\author{
  Yves Achdou\thanks { Universit{\'e} Paris Cit{\'e} and  Sorbonne Universit{\'e}, CNRS, Laboratoire Jacques-Louis Lions, (LJLL), F-75006 Paris, France, achdou@ljll-univ-paris-diderot.fr}
  \and
  Claude Le Bris\thanks{{\'E}cole des Ponts and INRIA, 6-8 avenue Blaise Pascal, Cit{\'e} Descartes, Champs-sur-Marne, 77455 Marne La Vall{\'e}e, France, claude.le-bris@enpc.fr}
  }

%%%%%%%%%%%%%%%%%%%%%%%%%%%%%%%%%%%%%%%%%%%%%%%%%%%%%%%%%%%%%%%%%
\maketitle
\begin{abstract}
We study homogenization  for a class of stationnary Hamilton-Jacobi equations in which the  Hamiltonian is obtained by perturbing  near the origin an otherwise periodic Hamiltonian. We prove that the limiting problem consists of a Hamilton-Jacobi equation outside the origin, with the same effective Hamiltonian as in periodic homogenization, supplemented at the origin with an effective  Dirichlet condition  that keeps track of the perturbation. Various comments and extensions are discussed.
\end{abstract}

\section{Introduction}
\label{sec:introduction}
This work discusses the homogenization limit for a first order stationary Hamilton-Jacobi equation of the form 
\begin{equation*}
\alpha \, u_\varepsilon+ H\left(\frac x \varepsilon, D u_\varepsilon\right)=0\quad \hbox{ in }\R^d,
\end{equation*}
and some related equations, where the Hamiltonian~$H$, besides satisfying some classical assumptions,  is obtained by a \emph{local}, compactly supported, perturbation of a periodic Hamiltonian~$\Hper$. We henceforth refer to such a perturbation as a local \emph{defect}.

Homogenization theory in the presence of local defects within an otherwise periodic environment was first introduced in~\cite{Milan}, in the first of a series of works by the second author, in collaboration with X.~Blanc and P-L.~Lions. It was further developed in~\cite{CPDE-2015,CPDE-2018} and other subsequent works by various authors, considering different classes of defects such as, in particular, interfaces between two different periodic media.  In those works, the typical setting  is that of a linear non-degenerate elliptic equation, first in divergence form and next in more general form. Only recently, some quasilinear equation was considered in~\cite{Wolf}. 

On the other hand, extensions to the context of nonlinear Hamilton-Jacobi equations arising from optimal control theory were  first addressed in a series of lectures~\cite{PLL-college} at Coll{\`e}ge de  France by P-L. Lions, where some results obtained in collaboration with P. Souganidis, ~\cite{PLL}, were described. Since the present contribution is very much related to \cite{PLL-college,PLL}, we will devote a special paragraph below to commenting our results in view of the latter references  and to stressing the similarities and differences.
In \cite{C-LB-S},  P.~Cardaliaguet, P.~Souganidis and the second author considered the extension to a fully nonlinear equation of a specific \emph{probabilistic} setting  originally introduced,  in the linear elliptic setting and from a computational perspective,  in~\cite{Anantha-LB}. In that setting, the defect is assumed to occur only with a small probability, and the homogenized limit is identified at the leading order in probability. 
More generally, an introductory account of the entire mathematical endeavor together with a large variety of possible applications and extensions, has appeared in~\cite{Notices}.

\medskip

In a distinct line of works, the first author, in collaboration with N.~Tchou in particular, studied
the asymptotic behaviour of the solutions to some classes of Hamilton-Jacobi equations depending on a vanishing parameter, for which the limit 
is characterized by boundary value problems with transmission conditions involving  effective Hamiltonians. In particular, 
\cite{MR3299352} dealt with dimension reduction for Hamilton-Jacobi equations posed in thin domains converging to networks. When the thickness of the domain vanishes, one finds  at the limit a Hamilton-Jacobi equation posed on the network with a special transmission condition at the vertices. The latter involves a so-called flux limiter which keeps track of the microscopic geometry of the thin domains.
Subsequently, the works \cite{MR3565416} and \cite{MR3912640} were concerned with  Hamilton-Jacobi equations  in an environment consisting of two different homogeneous
media separated by an oscillatory interface. The oscillations of the interface have small period and amplitude, and as both the latter parameters vanish,
the interface tends to an hyperplane. At the limit, one finds on the flat interface an effective nonlinear  transmission condition  keeping memory of the previously mentioned microscopic oscillations. Similarly, G. Galise, C. Imbert and R. Monneau studied in  \cite{MR3441209} a family of time dependent Hamilton-Jacobi equations 
on the simplest possible network composed of two half-lines with a perturbation of the Hamiltonian localized in a small region close to the junction.  Related homogenization problems
with applications  to traffic flows were discussed in \cite{MR3809148,MR4450245} by N. Forcadel and his coauthors.
Key arguments in \cite{MR3299352, MR3565416, MR3912640, MR3441209} are the construction of families of correctors
 that account for the localized pertubations of the environment and are defined in unbounded domains. We will see that the construction of such correctors  also plays a key role in the present work.
Note also that another common feature between the references  \cite{MR3299352, MR3565416, MR3912640, MR3441209} is that at the limit, the new effective transmission conditions are all posed on manifolds of codimension one. The study of such homogenization problems was indeed part of  the significant research effort that took place in the last decade
on the analysis of Hamilton-Jacobi equations posed on heterogeneous structures such as networks
\cite{MR3621434,MR3358634, MR3556345}, {\sl booklet}-like geometries or multidimensional junctions \cite{barles2013bellman,MR3690310, oudet2015equations}.
%and stratified media.
These problems all involve Hamilton-Jacobi equations on manifolds with nonlinear transmission conditions on submanifolds with codimension one.
%, or nested generalizations in the case of stratified media.
In sharp contrast, in the present work, the problem arising at the limit involves a boundary condition posed at the origin, which is generally not a manifold of codimension one. This causes technical difficulties and a priori prevents one from relying on the theories contained in \cite{MR3621434,MR3358634, MR3556345}.
On the contrary, the problem arising at the limit falls into the class of {\sl stratified} problems (and is a particularly simple example of such problems) introduced by A. Bressan and coauthors, see \cite{MR2291823} and later studied by G. Barles and E. Chasseigne,  see \cite{barles2018illustrated}. By and large, these problems involve Hamilton-Jacobi equations with discontinuities in a more general geometric setting, namely {\sl stratifications} of $\R^d$.
The  latter references contain in particular  comparison principles for viscosity solutions of stratified problems, but, up to our knowledge,  there is no literature yet on homogenization or singular perturbations leading to stratified problems, and in particular no proofs of convergence.

Finally, in connection with \cite{MR3565416, MR3912640},  the present contribution can be seen as  a first step toward homogenization theory for Hamilton-Jacobi equations with two different possibly periodic media separated by an  interface presenting localized defects of periodicity. We plan to address this aspect in a forecoming work.

\medskip

As already mentioned, the present study (essentially) considers a single localized defect, supported say in a neighborhood of  the origin,  within a periodic medium.
It identifies the homogenized limit in the case of a \emph{first order} Hamilton-Jacobi equation arising from optimal control, thus with a convex Hamiltonian; we will assume that the Hamiltonian has agreeable mathematical properties, which  will be made precise in the next section.
It is natural to expect that  at the limit, except possibly at the origin, the Hamilton-Jacobi equation involving the effective Hamiltonian obtained from periodic homogenization is satisfied.  We will see that if the defect has the effect of attracting the optimal trajectories starting not too far from the origin, (which occurs  when  the running cost or potential of the optimal control problem displays a prominent \emph{downward} bump within the periodic environment), then the homogeneity of the  environment is broken in the homogenized limit. We will indeed prove that, in this case, the limit problem involves a Dirichlet boundary condition at the origin. Besides, it is well known that, in singular domains such as $\R^d \setminus\{0\}$, uniqueness does not hold for the now classical formulation of Dirichlet problems due to H. Ishii, see e.g. \cite[chapter V, section 4]{MR1484411}. By contrast, uniqueness holds for the formulation of the Dirichlet problem that we find here from homogenization.
% The presence of this defect, located, say, at the origin and assumed of the "right form", that is, loosely speaking, modeling a prominent \emph{downward} bump within the periodic environment, breaks the asymptotic homogeneity of the limit environment. In the homogenized limit, it is indeed established in the present work that the limit is of the same type as that for the works just cited above, namely a \emph{stratified} problem. 
We stress again  that because of the effective Dirichlet condition, the limit is \emph{generally not} an Hamilton-Jacobi equation on the whole ambient space.
% It is rather such an equation \emph{outside} the origin, supplied with a specific type of Dirichlet \emph{boundary} condition at the origin.
%In the later sections of the present work, we will also consider the case of a \emph{randomized} defect as in some of the works in the linear and nonlinear settings recalled above.

\medskip

We wish to emphasize that the results exposed here are to be considered in the light of those obtained earlier by P-L.~Lions and  P.~Souganidis and summarized in~\cite{PLL-college} (and, likewise, of~\cite{C-LB-S} for randomized defects). In the control theoretic interpretation of that work, the typical local perturbation of the running cost is a bump oriented \emph{upwards} so that  the neighborhood of the origin becomes repulsive.
In~\cite{PLL-college}, the presence of such a  defect indeed does not affect the homogenized limit,
but only possibly \emph{"the next order correction"}, that is the definition of the corrector function itself.
In sharp contrast, if the defect makes the origin attractive,
then it  affects the homogenized limit itself. Note that it is not unexpected that signs, in the wide acceptation of that term, play a critical role in the context of nonlinear  equations.

\medskip

In the authors' view, the present work, together with the results of~\cite{PLL-college},    lay the groundwork for  a complete theory for Hamilton-Jacobi equations in periodic environments with local defects. Many challenging questions, some of them presumably quite difficult, however remain unsolved.
%The final section of the present article will further elaborate on this.

\medskip

Our article is more precisely organized as follows.  In the rest of this introduction, we first make precise, in Section~\ref{sec:setting}, the mathematical setting and assumptions of the problem we study. The Hamilton-Jacobi equation under consideration reads as~\eqref{eq:7} below. We next state and comment in Section~\ref{sec:main-result} our main result, namely Theorem~\ref{sec:main-result-1} which establishes for~\eqref{eq:7} the homogenized limit~\eqref{eq:11} through \eqref{eq:14} made precise therein. Section~\ref{sec:homog-refeq:7} then contains the detailed proof of Theorem~\ref{sec:main-result-1}. The proof essentially falls in four steps, respectively contained in Sections~\ref{sec:ergod-const-assoc} through~\ref{proof_of_main_th}. The final Section~\ref{sec:extension} contains several comments upon and extensions of the problem considered. First, in Section~\ref{sec:1D}, we illustrate in the one-dimensional setting  the general results of Section~\ref{sec:main-result}. We in particular consider, in Section~\ref{ssec:proba}, a probabilized variant of our problem inspired by some of the earlier works recalled above. Section~\ref{ssec:others} discusses the extension to higher dimensions of
the  considerations and results given in Section \ref{sec:1D}.

\section*{Acknowledgments}
The first author is partially on academic leave at Inria for the years 2021-22 and 2022-23 and acknowledges the hospitality of this institution during this period.
This research was partially supported by ANR (Agence Nationale de la Recherche) through project COSS ANR-22-CE40-0010-03. The research of the second author is partially supported by ONR under Grant N00014-20-1-2691 and by EOARD under Grant FA8655-20-1-7043. The final stage of the writing of this manuscript has been completed ultimately when the second author was visiting the Math+ Institute in Berlin. Partial financial support from the Deutsche Forschungsgemeinschaft (DFG, German Research Foundation) under Germany's Excellence Strategy - The Berlin Mathematics Research Center MATH+ (EXC-2046/1, project ID: 390685689) - is thus gratefully acknowledged.

The first author would like to thank N. Tchou for helpful discussions.

\subsection{Setting and assumptions}\label{sec:setting}

Let us define the problem and give the assumptions that will be used in the whole paper.

We consider Hamilton-Jacobi equations linked to infinite horizon optimal control problems in $\R^d$. 
The Hamiltonian $H: \R^d\times \R^d \to \R$ is of the form 
\begin{equation}
  \label{eq:1}
H(x,p)= \max_{a\in A } \Bigl(-p\cdot f(x,a)-\ell(x,a)\Bigr).
\end{equation}

Here, $A$ is a compact metric space,  $f: \R^d\times A\to \R^d$ is  a  bounded and continuous function,  Lipschitz continuous with respect to its first variable uniformly with respect to its second variable, i.e. 
 for any  $x,y\in \R^d$ and $a\in A$,
  \begin{displaymath}
    |f(x,a)-f(y,a)|\le L_f |x-y|,
  \end{displaymath}
 for some positive constant $L_f$.  We also suppose that there exists some radius $r_f>0$ such that for any $x\in \R^d$, $\{f(x,a),a\in A\}$ contains the ball $B_{r_f}(0)$, which implies that the trajectories are locally strongly controllable.

  We suppose that the function $\ell: \R^d\times A\to \R$ is  bounded and continuous and that here is a modulus of continuity $\omega_\ell$ 
such that for any  $x,y\in \R^d$ and $a\in A$, 
\begin{displaymath}
|\ell(x,a)-\ell(y,a)|\le \omega_{\ell} (|x-y|).
\end{displaymath}
 Define  $M_f=     \sup_{x\in \R^d, a \in    A} | f(x,a)| $ and $M_\ell=     \sup_{x\in \R^d, a \in    A} | \ell(x,a)| $. It is easy to check that the Hamiltonian $H$ defined in \eqref{eq:1} has the following properties: $H$ is convex with respect to its second argument, and  for any  $x,y,p,q\in \R^d$,
 \begin{eqnarray}
   \label{eq:2} H(x,p)&\ge& r_f |p| - M_\ell,\\
   \label{eq:3}|H(x,p)-H(y,p)|&\le& L_f|p||x-y| + \omega_{\ell} (|x-y|),\\
   \label{eq:4} |H(x,p)-H(x,q)|&\le& M_f |p-q|.
 \end{eqnarray}
Property (\ref{eq:2}) implies the coercivity of $H$ w.r.t. its second variable uniformly in its first variable, i.e.  $\lim_{|p|\to \infty} \inf_{x\in \R^d} H(x,p)=+\infty$.

We next assume that,  except in a neighborhood of the origin, the  dynamics $f$ and cost $\ell$ coincide with periodic functions. Let  $\T^d =\R^d /\Z^d$ denote the torus of $\R^d$. We assume that there exist $R_0>0$, $\fper:  \T^d \times A\to \R^d$ and $\ellper: \T^d \times A\to \R$ such that, if $|x|>R_0$, then $f(x,a)=\fper(x,a)$ and $\ell(x,a)=\ellper(x,a)$. Hence, the suprema $M_f$ and $M_\ell$ introduced above are indeed maxima. Let $\Hper$ stand for the related periodic Hamiltonian:
\begin{equation}
  \label{eq:5}
\Hper(x,p)= \max_{a\in A } \Bigl(-p\cdot \fper(x,a)-\ellper(x,a)\Bigr).
\end{equation}

Let $\alpha$ be a positive scalar and $\varepsilon$ be a small positive parameter
that will eventually vanish. 
It is well known, see e.g. \cite{MR1484411}, that the value function $u_\varepsilon$ of the following optimal control problem:
\begin{equation}\label{eq:6}
  u_\varepsilon(x)= \inf  \int_0 ^\infty    e^{-\alpha t} \ell \left( \frac {y(t)} \varepsilon, a(t)\right) dt \quad  \hbox{subject to } \left\{
    \begin{array}[c]{l}
        y(t)= x+\int_0 ^t f\left ( \frac {y(\tau)}\varepsilon ,a(\tau)\right) d\tau,\\
        a\in L^\infty(\R_+)\\
        a(t)\in A, \; \hbox{a.e.}
    \end{array}
      \right.
\end{equation}
is the unique viscosity solution in $\rm{BUC}(\R^d)$ of 
\begin{equation}
  \label{eq:7}
\alpha \, u_\varepsilon+ H\left(\frac x \varepsilon, D u_\varepsilon\right)=0\quad \hbox{ in }\R^d.
\end{equation}
Our goal is to study the asymptotic behaviour of $u_\varepsilon$ as $\varepsilon\to 0$. 
 \\
Homogenization of the periodic Hamilton-Jacobi equation involving $\Hper$ instead of $H$ is well understood since the pioneering work \cite{LPV}.
In the periodic case, the homogenized equation  is 
\begin{equation}
  \label{eq:8}
\alpha \, u+\overline H (Du)=0 \quad \hbox{ in } \R^d,
\end{equation}
 where the effective Hamiltonian is characterized as follows:
 for any $p\in \R^d$, $\overline H (p)$ is the unique real number such that there exists a periodic corrector $\chiperp$, i.e. a viscosity solution $ \chiperp \in C(  \T^d)$
of 
\begin{equation}
  \label{eq:9}
\Hper(y, p+ D\chiperp)=\overline H(p) \quad \hbox{in }\R^d.
\end{equation}
In general,  the latter periodic corrector is not unique even up to the addition of a scalar constant.
It is well known that, under the assumptions made above, $\overline H$ is convex on $\R^d$, Lipschitz continuous and coercive.
% Concerning the homogenization of the problem of interest with a defect of periodicity near the origin, namely (\ref{eq:7}), 
% we know that  in $\R^d\setminus\{0\}$, the homogenized problem will not keep track of the  defect of periodicity because homogenization is a local process: 
Therefore, there exists a vector $p_0\in \R^d$, possibly non unique,  such that
\begin{equation}
  \label{eq:10}
 \overline H(p_0) =\min_{q\in \R^d} \overline H(q).
\end{equation}
We choose such a vector $p_0$ and fix it for all what follows.

\subsection{The main result}\label{sec:main-result}

Our main result is the following:
\begin{theorem}
  \label{sec:main-result-1}
We consider the solution $u_\varepsilon$ of (\ref{eq:7}).
As $\varepsilon\to 0$, the family $u_\varepsilon$ converges locally uniformly to the unique function $u\in \rm{BUC}(\R^d)$ defined by the following properties:
\begin{enumerate}
\item $u$ is a viscosity solution of
 \begin{equation}
\label{eq:11}
 \alpha \, u+\overline H(Du)= 0\quad  \hbox{in }\R^d \setminus \{0\},
\end{equation}
with the effective Hamiltonian $\overline H$  defined above by periodic homogenization, see (\ref{eq:9})
%and remember (\ref{eq:54}) 
\item  The condition
  \begin{equation}
    \label{eq:12}
    \alpha \, u(0)+E\le 0
  \end{equation}
holds, where $E$ is the  effective Dirichlet datum  defined in paragraph~\ref{sec:ergod-const-assoc} below. In addition, if $\phi\in C^1(\R^d)$ is such that $u-\phi$ has a local maximum at the origin, then 
  \begin{equation}\label{eq:13}
 \alpha \, u(0)+\overline H(D\phi(0))\le 0.
  \end{equation}
\item If $\phi\in C^1(\R^d)$ is such that $u-\phi$ has a local minimum at the origin, then 
  \begin{equation}\label{eq:14}
 \alpha \, u(0)+\max\left( E, \overline H(D\phi(0))\right)\ge 0.
  \end{equation}
\end{enumerate}
\end{theorem}
% \begin{remark}\label{sec:main-result-2}
%   Note that Theorem \ref{sec:main-result-1} could also be formulated without supposing (\ref{eq:54}) by suitably shifting the problem.
% \end{remark}
\begin{remark}\label{sec:main-result-3}
  The conditions  \eqref{eq:11} through \eqref{eq:14} may be seen as a weak formulation of a Dirichlet boundary value problem comprising the  Hamilton-Jacobi equation  \eqref{eq:11} in the singular open set $\R^d\setminus \{0\}$ and the Dirichlet boundary condition  $u(0)=-E/\alpha$. Note that this formulation is stronger than that due to H. Ishii, which expresses that at the origin,  $u$ is both a viscosity subsolution of $\alpha \, u(0) + \min (\overline H(D u(0)), E)\le 0$ and a  viscosity supersolution of $\alpha \, u(0) + \max (\overline H(D u(0)), E)\ge 0$, see for example \cite[Chapter V, Section 4]{MR1484411}, \cite{MR3135340} and \cite{MR3135341}.
  Indeed, while the supersolution condition \eqref{eq:14} coincides with that of Ishii,
  the subsolution condition  \eqref{eq:12}-\eqref{eq:13} is stronger.\\
  A consequence of the difference between the two formulations regards uniqueness.  Reference  \cite[Chapter V, Section 4]{MR1484411} contains  a counterexample  for the  uniqueness of  viscosity solutions in the sense of Ishii of Dirichlet problems  posed in  $\R^d\setminus \{0\}$. Note also that an interior cone condition is  sufficient  for uniqueness,  see (4.30) in the latter reference, but this condition is not satisfied by  $\R^d\setminus \{0\}$. By contrast, and as we will see in the proof of Theorem  \ref{sec:main-result-1} in Section \ref{proof_of_main_th} below, uniqueness holds for  \eqref{eq:11}, \eqref{eq:12}, \eqref{eq:13}, \eqref{eq:14}.
\end{remark}
\begin{remark}\label{sec:main-result-4}
  % The conditions  \eqref{eq:11}--\eqref{eq:14} are 
  % reminiscent of the notion of stratified solutions of Hamilton-Jacobi equations, see \cite{barles2018illustrated}. However, in dimension  $d>1$, 
  % the  partition of $\R^d$ into the sets $\R^d\setminus \{0\}$ and $\{0\}$ is not an admissible stratification as defined in the latter reference. The theory in \cite{barles2018illustrated} therefore does not apply.
  The conditions  \eqref{eq:11} through \eqref{eq:14} fall into the general notion of stratified solutions of Hamilton-Jacobi equations, see \cite{MR2291823,barles2018illustrated}, and correspond to  the  partition of $\R^d$ into the sets $\R^d\setminus \{0\}$ and $\{0\}$. The final step in the proof of Theorem  \ref{sec:main-result-1} needs a comparison result which can be found in  \cite{barles2018illustrated} in a much more general setting.
  In order to keep the paper self-contained, we will give a short proof of the comparison principle, see Section \ref{proof_of_main_th},  because it is quite simple in the particular case under consideration.
\end{remark}

\begin{remark}\label{sec:main-result-6}
 With the same assumptions on $H$, Theorem \ref{sec:main-result-1} may be easily generalized to the homogenization of Hamilton-Jacobi equations of the form
    \begin{displaymath}
      \alpha \, u_\varepsilon + H\left(\frac x \varepsilon, D u_\varepsilon\right)= b(x)\quad \hbox{ in }\R^d,
    \end{displaymath}
    where $b\in \rm{BUC}(\R^d)$. The effective Dirichlet datum $E$ would then  depend on $b$ only through its  value  at the origin.
    \\
    Similarly, the result can  be easily extended to problems of the form 
    \begin{displaymath}
      \alpha\, u_\varepsilon + H_\varepsilon(x,Du_\varepsilon)=0 \quad \hbox{in } \R^d,
    \end{displaymath}
    where $ H_\varepsilon(x,p)= \max_{a\in A}  \left( -p\cdot f_\varepsilon(x,a)  - \ell_\varepsilon(x,a) \right)$ and
      \begin{eqnarray*}
          \ell_\varepsilon(x,a)= \ell_{1} (x, \frac x \varepsilon, a) + \ell_0(\frac x \varepsilon, a), \\
          f_\varepsilon(x,a)= f_{1} (x, \frac x \varepsilon, a) + f_0(\frac x \varepsilon, a), 
      \end{eqnarray*}
the functions $ \ell_0$, $f_0$ are as above (they model the perturbation located near the origin) and $ \ell_{1}$ and $f_1$ are smooth with respect to their first argument and periodic with respect to the second one.
 
Other generalizations to different settings may be considered. While those to evolutive problems seem fairly easy,  extensions to {\sl viscous } Hamilton-Jacobi equations seem more challenging.
\end{remark}

\begin{remark}\label{sec:main-result-5}
  As mentioned in the introduction, the homogenization of periodic Hamilton-Jacobi equations in the presence of a defect has been first addressed by P-L. Lions and P. Souganidis, see the results announced in \cite{PLL-college,PLL}. The assumptions made   therein  imply that $E\le \min_{p\in \R^d} \overline H(p)$ and that the defect does not show up in the  homogenized limit, i.e.  $u$ is a viscosity solution of $\alpha \, u + \overline H(Du)=0$ in $\R^d$. We therefore recover as a particular case of our setting the homogenized limit obtained in  \cite{PLL-college,PLL}, under the particular assumptions made there.

For example, in the case when $p_0=0$ in \eqref{eq:10}, we find that $u$ is then 
identically equal to $-\frac 1 \alpha \overline H (0)=-\frac 1 \alpha \min_{p} \overline H (p)$. 

  On the contrary, if
  \begin{equation}
    \label{eq:15}
    E> \overline H (0)=   \min_{p\in \R^d} \overline H(p),
  \end{equation}
then  the defect is  visible in the homogenized limit. Indeed, it can first be proven (see
\S~\ref{sssec:defaut-infini} below) that
 \begin{equation}
    \label{eq:16}
  \lim_{|x|\to \infty} u(x)=-\frac 1 \alpha \overline H (0).
\end{equation}
Combining \eqref{eq:12},  \eqref{eq:15} and \eqref{eq:16}, we then see that $u$ is not a constant function, so it differs from the homogenization limit in the absence of defect.
We also claim that in this particular case, the global infimum of~$u$ is reached at the origin and only at the origin. Indeed,  \eqref{eq:12}, \eqref{eq:15} and \eqref{eq:16} altogether imply that there exists a minimizer $x_0$ of $u$. If $x_0\not=0$, then  \eqref{eq:11} yields that $ \alpha\,u(x_0)\,\ge -{\overline H}(0)$,  thus $ \alpha\,u(x_0)>-E\ge  \alpha\,u(0)$, and we reach a contradiction.

% To summarize, Theorem  \ref{sec:main-result-1} generalizes the result of P-L. Lions and P. Souganidis, because it also covers the case when the homogenized problem  involves a Dirichlet boundary condition which is caused by the defect.

  Note also that,  in  \cite{PLL-college},  thanks to the assumption made,  the authors are able to construct correctors $\chi_p$ in the whole space $\R^d$  for all vectors $p\in \R^d$ ($p$ stands for the gradient of a smooth test-function at the origin).  Their construction  uses arguments from the theory of optimal control. These correctors, which differ from those used in periodic homogenization, can  be used for proving the convergence to the homogenized problem by means of the now classical method of perturbed test-functions, see \cite{MR1007533}.
 
  In contrast, when $E>\min_{p\in \R^d} \overline H(p)$,  we will have to adjust the  strategy proposed in   \cite{PLL-college} in  order to keep relying on the theory of optimal control : using arguments somewhat reminiscent of those proposed in \cite{MR3299352},  we will  construct correctors associated to suitable piecewise affine functions rather than smooth or linear functions. 
\end{remark}

\section{Homogenization of (\ref{eq:7}): the proof of Theorem~\ref{sec:main-result-1}}
\label{sec:homog-refeq:7}
Recall that our goal is to understand the asymptotic behavior of $u_\varepsilon$ as $\varepsilon$ tends to $0$. 
First, using  either comparison principles, see for example \cite[Chapter II, Theorem 3.5]{MR1484411} or arguments from the theory of optimal control, we see  that
\begin{displaymath}
-\max_{y\in \R^d} H(y,0) \le \alpha \, u_\varepsilon(x) \le -\min_{y\in \R^d} H(y,0).  
\end{displaymath}
 From this estimate and \eqref{eq:7}, we infer from the coercivity of the Hamiltonian  that $u_\varepsilon$ is Lipschitz continuous in $\R^d$ with a Lipschitz constant independent of $\varepsilon$.

 In order to study the asymptotic behaviour of $u_\varepsilon$, we consider 
 \begin{eqnarray}
   \label{eq:17}
 \overline{u}(x)=\limsup_{\varepsilon\to 0}u_\varepsilon(x),
\\
\label{eq:18}
 \underline{u}(x)=\liminf_{\varepsilon\to 0}u_\varepsilon(x).
 \end{eqnarray}
 Note that, from the observation above on the regularity of $u_\varepsilon$,
$ \overline{u}$ and  $\underline{u}$ coincide respectively with the half-relaxed semi-limits  $\ds \limsup_{x'\to x, \varepsilon\to 0}u_\varepsilon(x')$ and $\ds \liminf_{x'\to x, \varepsilon\to 0}u_\varepsilon(x')$, that are classically  used in the homogenization of Hamilton-Jacobi equations. It is clear that the functions $ \overline{u}$ and $ \underline{u}$  are  bounded and Lipschitz continuous.
%  it is now classical to consider the half-relaxed semi-limits
% \begin{eqnarray*}
%  \overline{u}(x)&=&{\limsup_{\varepsilon\to 0}}^{*} u_\varepsilon(x)=\limsup_{x'\to x, \varepsilon\to 0}u_\varepsilon(x'),
% \\ 
%  \underline{u}(x)&=&\underset{\varepsilon\to 0}{{\liminf}_{*}}u_\varepsilon(x)=\liminf_{x'\to x, \varepsilon\to 0}u_\varepsilon(x').
% \end{eqnarray*}
% From the observation above on the regularity of $u_\varepsilon$, 
%  the functions $ \overline{u}$ and $ \underline{u}$  are well defined, bounded and Lipschitz continuous, and 
%  \begin{eqnarray}
%    \label{eq:10b}
%  \overline{u}(x)=\limsup_{\varepsilon\to 0}u_\varepsilon(x),
% \\
% \label{eq:11b}
%  \underline{u}(x)=\liminf_{\varepsilon\to 0}u_\varepsilon(x).
%  \end{eqnarray}
% %\subsection{Asymptotic behavior in $\R^d\setminus\{0\}$}
% %\label{sec:asymptotic_behavior}

We now observe  that, in $\R^d\setminus\{0\}$, the homogenized equation does not keep track of the  defect of periodicity and is identical to (\ref{eq:8}).
This is  not surprising,  since the support of the defect shrinks to $\{0\}\times A$ as $\varepsilon\to 0$,
 while the related Hamiltonian  does not become singular. 
\begin{proposition}
\label{sec:homog-refeq:7-1}
 The functions $\overline{u}$ and $\underline{u}$ are respectively a bounded subsolution and a bounded supersolution of (\ref{eq:11})  in $\R^d \setminus \{0\}$.
\end{proposition}
\begin{proof}
The proof is classical and relies on perturbed test-functions techniques, see \cite{MR1007533}. 
\end{proof}

\paragraph{The strategy of proof  for Theorem~\ref{sec:main-result-1}}
Below, the remainder of the proof of Theorem~\ref{sec:main-result-1}  is done in four different steps. Accordingly,  Section~\ref{sec:homog-refeq:7} is cut into four parts.
The four  steps, which will be rapidly summarized below, mostly rely on the theory of viscosity solutions to first order Hamilton-Jacobi equations, 
except the third one, which  also contains  arguments from the theory of optimal control.
\begin{enumerate}
\item The first step consists of constructing the ergodic constant $E$ associated to the defect and a related  corrector $w$.
It will be  proved that $E\ge \overline H(p_0)$. An important difficulty is that the corrector $w$  must be a function defined in the whole space $\R^d$, which makes it necessary to impose some  condition at infinity.
 We will see that the latter amounts to the fact that  $w$ is the locally uniform limit as $R\to +\infty$ of a family $(w^R)_{R>0}$ of solutions of
problems with state constraints posed in the balls $B_R(0)$.  
From the optimal control theory viewpoint,  these problems,  refered to as {\sl truncated cell problems}, account for trajectories that 
remain close to the defect at the microscopic scale.  Proposition~\ref{sec:ergod-const-glob-2} below contains information on the growth of $w$ at infinity.
\item In the second step, we prove  that  the upper-limit $\overline u$ satisfies conditions (\ref{eq:12}) and (\ref{eq:13}). While  (\ref{eq:13})  is a rather easy consequence of Proposition \ref{sec:homog-refeq:7-1}, the proof of (\ref{eq:12}) relies on
Evans' method of perturbed test-functions, see \cite{MR1007533}.  The construction of the perturbed test-function involves the above mentioned solution $w^R$ to the  truncated cell problem in the ball $B_R(0)$.
\item The third step, the most involved one, consists of proving that the lower-limit $\underline u$ satisfies condition (\ref{eq:14}). Most of the work concerns the situation in which $E> \overline H(p_0)$. Here again, one needs to construct a suitable perturbation of a test-function by a corrector. The crucial theoretical novelty is the choice of the test-function. Instead of using smooth or affine test-functions as it is classically done, we consider  suitable piecewise affine test-functions, in the spirit of \cite{MR3299352}. Their behaviour at infinity  makes it possible to construct associated correctors in the whole space $\R^d$.
 The key intermediate result is Proposition~\ref{sec:function--underline-1} which states the existence of such correctors with a strictly sublinear growth at infinity.
Its proof  uses as an important ingredient the function $w$ introduced in the first step and Proposition~\ref{sec:ergod-const-glob-2}.
 Note also that the proof of Proposition~\ref{sec:function--underline-1} relies on control theoretical arguments in the same spirit as  in \cite{PLL-college}.
The choice of the above mentioned piecewise affine test-functions is precisely made  so that the latter arguments can be applied when $E>  \overline H(p_0)$.
\item The fourth step of the proof mostly consists of deducing from the previously obtained results that $\underline u=\overline u$, by means of a comparison principle. Once this is done, the results announced in Theorem~\ref{sec:main-result-1} follow.
\end{enumerate}

\subsection{The ergodic constant associated to the defect}
\label{sec:ergod-const-assoc}
 We next introduce ingredients which play a key role in the effective boundary condition at the origin.
\subsubsection{Ergodic constants for state-constrained problems in truncated domains}
\label{sec:ergod-const-state}
In order to understand the effect of the defect on the asymptotics of $u_\varepsilon$, we start by solving {\sl truncated cell problems} in balls centered at the origin;
these are associated to  state constrained boundary conditions. From the optimal control theory viewpoint,  these problems account for trajectories that 
remain close to the defect at the microscopic scale.\\
For $\lambda>0$, $R>0$, we know from e.g. \cite{MR838056,MR951880} that  there exists a unique function $w^{\lambda, R} \in C(\overline{B_R(0)})$ such  that
\begin{eqnarray}
  \label{eq:19}
\lambda w^{\lambda, R} + H(y, Dw^{\lambda, R}) &\le& 0 \quad \hbox{ in } B_R(0),\\
\label{eq:20}
\lambda w^{\lambda, R} + H(y, Dw^{\lambda, R}) &\ge& 0 \quad \hbox{ in } \overline {B_R(0)},
\end{eqnarray}
the above inequalities being understood in the sense of viscosity. An equivalent way to write (\ref{eq:19})-(\ref{eq:20}) is the following: 
\begin{eqnarray}
\label{eq:21}
\lambda w^{\lambda, R} + H(y, Dw^{\lambda, R}) &=& 0 \quad \hbox{ in } B_R(0),\\
\label{eq:22}
\lambda w^{\lambda, R} + H^{\uparrow}(y, Dw^{\lambda, R}) &=& 0 \quad \hbox{ on } \partial B_R(0),
\end{eqnarray}
in the sense of viscosity, where, for $y\in \partial B_R(0)$,
\begin{equation}
\label{eq:23}
H^{\uparrow}(y,p)= \max_{a\in A, \;f(y,a)\cdot y \le 0 } \Bigl(-p\cdot f(y,a)-\ell(y,a)\Bigr).
\end{equation}
is the Hamiltonian associated to the admissible trajectories that do not exit the ball $B\overline{(0,R)}$ through $y$. The function $w^{\lambda, R}$ is the value function of the following infinite horizon state constrained optimal control problem in $\overline{B_R(0)}$,
\begin{equation}\label{eq:24}
  w^{\lambda, R}(z)= \inf  \int_0 ^\infty    e^{-\lambda t} \ell \left( \frac {y(t)} \varepsilon, a(t)\right) dt \quad  \hbox{subj. to } \left\{
    \begin{array}[c]{l}
        y(t)= z+\int_0 ^t f\left ( \frac {y(\tau)}\varepsilon ,a(\tau)\right) d\tau\\
        y(t)\in \overline{B_R(0)},\\
        a\in L^\infty(\R_+) \\ a(t)\in A, \; \hbox{a.e.}
    \end{array}
\right.
\end{equation}
which the reader may compare to \eqref{eq:6}.
Since $\ell$ is bounded on $\R^d \times A$
(because, in particular, $A$ is compact and $\ell$ coincides with $\ellper$ out of $B_{R_0}(0)$),  $\lambda  \|w^{\lambda, R}\|_{L^\infty(B_R(0))}$ is bounded uniformly in $\lambda$ 
and $R$. More precisely,
 $  \min_{(y,a)\in \R^d\times A} \ell(y,a) \le   \lambda  w^{\lambda, R} \le \max_{(y,a)\in \R^d\times A} \ell(y,a)$.
This and  the uniform coercivity of $H$ imply with \eqref{eq:21} that $\|Dw^{\lambda,R}\|_ {L^\infty(B_R(0))}  $ is bounded uniformly in $\lambda$ and $R$.

Using Ascoli-Arzel{\`a} theorem, we may suppose that up to the extraction of a sequence,  as $\lambda\to 0$,  $\lambda w^{\lambda, R}$ 
tends uniformly on $\overline {B_R(0)}$ to some {\sl ergodic} constant $-E^R$ which is bounded from above and below uniformly in $R$,   and that $w^{\lambda, R} -w^{\lambda, R}(0)$ tends uniformly 
on $\overline {B_R(0)}$ to some function $w^R$ such that $w^R(0)=0$ and which is Lipschitz continuous with a Lipschitz constant independent of $R$.  By classical results on the stability of viscosity solutions 
of state constrained problems,  $w^R$ is a viscosity solution of 
\begin{eqnarray}
\label{eq:25}
 H(y, Dw^{ R}) &\le& E^R \quad \hbox{ in } B_R(0),\\
\label{eq:26}
 H(y, Dw^{R}) &\ge& E^R \quad \hbox{ in } \overline {B_R(0)}.
\end{eqnarray}
The comparison principle for state constrained problems, see \cite{MR838056,MR951880}, yields the uniqueness 
of $E^R$ such that (\ref{eq:25})-(\ref{eq:26}) has a solution. Thus, 
$\lim_{\lambda\to 0} \| \lambda w^{\lambda, R}+ E^R\| _{C(\overline B_R(0))}=0$ (uniform convergence and not only for a subsequence).

\subsubsection{The ergodic constant and the cell problem}\label{sec:ergod-const-glob}
We deduce for example from (\ref{eq:24})  that 
\begin{displaymath}
  R_1\ge R_2 \quad \Rightarrow \quad \lambda w^{\lambda, R_1} \le \lambda w^{\lambda, R_2}, 
\end{displaymath}
and passing to the limit as $\lambda\to 0$, we obtain the monotonicity property of the ergodic constants $E^R$:
\begin{equation}
  \label{eq:27}
  R_1\ge R_2 \quad \Rightarrow \quad E^{ R_1} \ge E^{ R_2}. 
\end{equation}
Since $E^R$ is bounded from above independently of $R$
% by $-\min_{(y,a)\in \R^d\times A} \ell(y,a)$,
(\ref{eq:27})  implies that
\begin{equation}
   \label{eq:28}
   E=\lim_{R\to \infty} E^R
\end{equation}
 exists in $\R$.

Similarly, since  $w^R(0)=0$, $w^R$ is Lipschitz continuous on $\overline {B_R(0)}$ with a Lipschitz constant independent of $R$, we may construct by Ascoli-Arzel{\`a} theorem and  a diagonal extraction argument a sequence $(R_n)_{n\in \N}$,  $R_n\to +\infty$ as $n\to \infty$,  such that $w^{R_n}$ tends to some function $w$ locally uniformly in $\R^d$;  we then see that $w(0)=0$ and $w$ is a  Lipschitz continuous viscosity solution of 
\begin{equation}
  \label{eq:29}
H(y, Dw)= E \quad \hbox{in }\R^d.
\end{equation}
Let us now zoom out and pass to the macroscopic scale by considering the function $w_\varepsilon: x \mapsto \varepsilon w(\frac x \varepsilon)$; it is clearly a viscosity solution of $ H(\frac x \varepsilon , D_x w_\varepsilon)= E$, and it is Lipschitz continuous with the same constant as $w$. Hence, after the extraction of a sequence, we may assume that 
$w_\varepsilon$ converges locally uniformly to some Lipschitz function  $W$ on $\R^d$. As for Proposition \ref{sec:homog-refeq:7-1}, a standard argument yields  that $W$ is a viscosity solution of $\overline H (DW)=E$.  This implies the important inequality 
\begin{equation}
  \label{eq:30}
E\ge \min_{p\in \R^d }  \overline H (p) = \overline H(p_0),
\end{equation}
 the right hand side equality holding because of \eqref{eq:10}.
%where the identity comes from Assumption \ref{sec:setting-assumptions}.

Proposition  \eqref{sec:ergod-const-glob-2} below gives some information on the behaviour of $w$ at infinity. It will be useful for proving that $\underline u$ satisfies \eqref{eq:14}.
\begin{proposition}
  \label{sec:ergod-const-glob-2}
If $E> \overline H(p_0)= \min _{p\in \R^d} \overline H(p)$, then 
\begin{equation}\label{eq:31}
  \lim_{|y|\to +\infty} w(y)-p_0\cdot y =+\infty,
\end{equation}
where  $w$ is the corrector constructed in (\ref{eq:29}).

If $E=\overline H(p_0)= \min _{p\in \R^d} \overline H(p)$, then $w(y)-p_0\cdot y$ is bounded from below.
\end{proposition}
\begin{proof}
$\;$
  \paragraph{Step 1}
Let us start by proving the desired results in the case when  $p_0=0$.
\\
 With $R_0$ defined in Section \ref{sec:setting}, we already know that, for all $R>R_0$, the corrector $w^R$ solution to (\ref{eq:25})-(\ref{eq:26}), is Lipschitz continuous with a Lipschitz constant $C>0$ independent of $R$.

 For $R_1>R_0$, let $Q_{R_1}$ be the cube 
  \begin{displaymath}
Q_{R_1} =\{y\in \R^d,\;\|y\|_{\infty} \le R_1\}.
\end{displaymath}
It contains $\overline{B_{R_0}(0)}$. It is clear that, for all $R>R_0$ and  $y\in \partial Q_{R_1}$,
\begin{equation}
  \label{eq:32}
|w^R(y)|\le C \sqrt{d} R_1,
\end{equation}
using the facts that $w^R(0)=0$ and that $C$ is a Lipschitz constant for $w^R$.
We distinguish two cases:
\begin{enumerate}
\item  $E> \overline H(0)$.
\\ Let $(e_i)_{i=1,\dots, d}$ be the canonical basis of $\R^d$ associated with the system of coordinates $(y_i)_{i=1,\dots, d}$.
Since %$\lim_{R\to \infty}  E^R=
$E>\overline H(0)$, we may choose  $\delta$,  $0<\delta< E-\overline H(0)$.
Now, because of the convergence of $E^R$ to $E$, we have for $R$ large enough (that we can always suppose larger than $ \sqrt{d} R_1$, so that the ball $B_R(0)$ contains $\overline{Q_{R_1}}$),
%and $R>R_1$ such that 
 \begin{equation}
   \label{eq:33}
E\ge E^R> E- \delta.
 \end{equation}
Since $\overline H(0)= \min _{p}\overline H(p)$, the continuity and the coercivity of $\overline H$ together with $E-\delta> \overline H(0)$ imply that for all $i \in \{1,\dots, d\}$, there exists   $(\underline{p}_{i,\delta}, \overline{p}_{i,\delta})\in \R^2$,   $\underline{p}_{i,\delta}<0< \overline{p}_{i,\delta}$ such that 
\begin{equation}
  \label{eq:34}
\overline H (\underline{p}_{i,\delta} e_i)= \overline H (\overline{p}_{i,\delta} e_i)= E-\delta.
\end{equation}
Consider now the functions 
\begin{eqnarray}
  \label{eq:35}
\overline w_{i,\delta}(y)= \overline c_{i,\delta}+\overline{p}_{i,\delta} y_i +\chi_{{\rm{per}},  \overline{p}_{i,\delta} e_i} (y),
\\
\label{eq:36}
\underline w_{i,\delta}(y)= \underline c_{i,\delta}+\underline{p}_{i,\delta} y_i +\chi_{{\rm{per}},  \underline{p}_{i,\delta} e_i}(y).
\end{eqnarray}
where $\overline c_{i,\delta}$ and $ \underline c_{i,\delta}$ are scalars chosen such that, for any $y\in \partial Q_{R_1}$,
\begin{equation}
  \label{eq:37}
\overline w_{i,\delta}(y) < -C \sqrt{d} R_1 \le   w^R(y)\quad \hbox{and}\quad \underline w_{i,\delta}(y) < - C \sqrt{d} R_1 \le  w^R(y),
\end{equation}
recalling that $\overline{Q_{R_1}}\subset B_R(0)$.
Moreover, since $H$ coincides with $\Hper$ in $(\R^d \setminus\overline{Q_{R_1}})\times \R^d$, we see from \eqref{eq:9} that 
\begin{eqnarray}\label{eq:38}
  H(y,   D\overline w_{i,\delta}(y) )= E-\delta ,\\\label{eq:39}
  H(y,   D\underline w_{i,\delta}(y) )= E-\delta ,
\end{eqnarray}
in the sense of viscosity in $\R^d \setminus \overline{Q_{R_1}}$.\\
Since $w^R$ is a viscosity supersolution of (\ref{eq:26}), we may use a comparison principle in  $\overline {B_R(0)}\setminus Q_{R_1}$
and deduce from~(\ref{eq:33}) and (\ref{eq:37})-(\ref{eq:38})-(\ref{eq:39}) that  for $R$ large enough and for all $y\in \overline {B_R(0 )}\setminus Q_{R_1}$,
\begin{equation*}
  \overline w_{i,\delta}(y) \le w^R(y), \quad\hbox{and }\quad   \underline w_{i,\delta}(y) \le w^R(y).
\end{equation*}
To summarize, we have proved that, for all   $y\in \overline B_R(0 )\setminus {Q_{R_1}}$:
\begin{equation}
  \label{eq:40}
\max \left\{\overline w_{i,\delta}(y),  \underline w_{i,\delta}(y) ,\; i=1,\dots, d\right\} \le  w^R(y).
\end{equation}
By passing to the limit in $R$ (possibly after the extraction of a sequence), we deduce that  for any $y\in \R^d\setminus Q_{R_1}$,
\begin{equation}
  \label{eq:41}
\max \left\{\overline w_{i,\delta}(y),  \underline w_{i,\delta}(y) ,\; i=1,\dots, d\right\} \le  w(y).
\end{equation}
Since the correctors $\chi_{{\rm{per}},  \overline{p}_{i,\delta} e_i}$ are bounded functions, it is clear from \eqref{eq:35}-\eqref{eq:36} that there exists a constant $c$ such that for all $y\in \R^d\setminus Q_{R_1}$,
\begin{equation*}
  \begin{split}
 & \max \left\{\overline w_{i,\delta}(y),  \underline w_{i,\delta}(y) ,\; i=1,\dots, d\right\} \\ \ge & \min  \left\{|\overline p_{i,\delta}|,  |\underline p_{i,\delta}| ,\; i=1,\dots, d\right\} \|y\|_{\infty} - c.    
  \end{split}
\end{equation*}
This yields the desired result, namely (\ref{eq:31}) in the case when $p_0=0$ and $E>\overline H(0)$.
\item $E=\overline H(0)$. For any $c$, the function  $\underline{w}(y)= \chi_{{\rm per},0} (y)+c$ thus satisfies
  \begin{displaymath}
    H(y, D\underline{w})= E, \quad \hbox{ in } B_R(0)\setminus \overline Q_{R_1},
  \end{displaymath}
in the sense of viscosity. Moreover, it is possible to choose
 $c$  such that  $ \underline w(y)=\chi_{{\rm per},0} (y)+c <  -C \sqrt{d} R_1$ for all $y\in \partial Q_{R_1}$. By a comparison principle, it follows that
$ \underline w\le w^R$ in $ \overline{B_R(0)}\setminus  Q_{R_1}$. 
By passing to the limit in $R$ (possibly after the extraction of a sequence), we deduce that $w$ is bounded from below, which concludes the proof of Proposition \ref{sec:ergod-const-glob-2} when $p_0=0$.
\end{enumerate}

  \paragraph{Step 2}
We need to prove the result for a general $p_0$.  The idea consists of  suitably shifting the Hamiltonians. More explicitly, we consider the new running  costs and Hamiltonians:
 \begin{eqnarray}\label{eq:42}
   \widetilde \ell (x,a)&=& \ell(x,a)+p_0\cdot f(x,a),\\\label{eq:43}
   \widetilde \ellper (x,a)&=& \ellper(x,a)+p_0\cdot \fper(x,a),\\\label{eq:44}
 \widetilde H(x,p)&= & H(x,p+p_0)=  \sup_{a\in A} \left( -p\cdot f(x,a)  -  \widetilde \ell (x,a)\right),\\\label{eq:45}
 \widetilde {\Hper}(x,p)&= & \Hper(x,p+p_0)=  \sup_{a\in A}  \left(-p\cdot \fper(x,a)  -  \widetilde \ellper (x,a)\right),
 \end{eqnarray}
 which satisfy all the  stuctural assumptions made on $\ell$, $\ellper$, $H$ and $\Hper$ in Section~\ref{sec:setting}.
 \\
 It is clear that the new effective Hamiltonian $\overline {\widetilde H}(p)= \overline H(p+p_0)$ 
 is obtained from $\widetilde {\Hper}$ by solving the shifted cell problem: 
 \begin{equation}
 \label{eq:46}
 \Hper(y, p_0+p+ D\widetilde \chiperp)=\overline {\widetilde H}(p) \quad \hbox{in }\R^d,
 \end{equation}
 the solutions of which are of the form $ \widetilde \chiperp= \chi_{\rm{per}, p+p_0}$, where $ \chi_{\rm{per}, p+p_0}$
  is a solution of the original periodic  cell problem (\ref{eq:9}) associated to $p+p_0$. Note that $0\in\argmin \overline {\widetilde H}$.\\
  It is straightforward to realize that $\widetilde w^R(y)= w^R(y)-p_0\cdot y$ is a viscosity solution of  the state constrained problem:
  \begin{eqnarray}
\label{eq:47}
\widetilde H(y, D\widetilde w^{ R}) &\le& E^R \quad \hbox{ in } B_R(0),\\
\label{eq:48}
 \widetilde H(y, D\widetilde w^{R}) &\ge& E^R \quad \hbox{ in } \overline {B_R(0)}.
\end{eqnarray}
and that $\widetilde w^R$ tends to $\widetilde w(y)= w(y) -p_0 \cdot y$ locally uniformly.
We can apply the results proven above to $\widetilde H$ and  $\widetilde w$ because $ \overline {\widetilde H}$ reaches its minimal value at $0$. This yields the desired result in the general case.
\end{proof}

\subsection{The function $ \overline  u$ is a subsolution of (\ref{eq:11}) and satisfies (\ref{eq:12})-(\ref{eq:13})}\label{sec:function--overline}

\begin{proposition}
  \label{sec:ergod-const-glob-1}
The upper limit $\overline u$ satisfies (\ref{eq:12}).
\end{proposition}
\begin{proof}
  Let us proceed by contradiction and assume that 
  \begin{equation}
    \label{eq:49}
 \alpha \overline u(0) +E = \theta >0.
  \end{equation}
 Using $w^R$ defined in paragraph~\ref{sec:ergod-const-glob} (recall  $w^R (0)=0$),
%and   (\ref{eq:25})-(\ref{eq:26})), let us define
we define
\begin{displaymath}
  \phi^{\varepsilon, R}= \overline u (0)+ \varepsilon w^R (\frac x \varepsilon).
\end{displaymath}
We deduce  from (\ref{eq:26}) that $ \phi^{\varepsilon, R}$ is a viscosity supersolution of 
 \begin{displaymath}
  \alpha \phi^{\varepsilon, R}(x) + H\left(\frac x \varepsilon, D \phi^{\varepsilon, R}  \right)\ge \alpha \varepsilon w^R\left(\frac x \varepsilon\right)+E^R-E+\theta\quad \hbox{in  } \overline{B_{\varepsilon R}(0)}.
\end{displaymath}
There exists $r>0$ such that  $E^R - E\ge -\frac \theta 4$ for any $R\ge r$. Let us fix such a value of $R$. 

Having fixed $R$, we see that  for $\varepsilon_0$ sufficiently small and any $\varepsilon$ such that $0<\varepsilon< \varepsilon_0$,
\begin{displaymath}
  \alpha \varepsilon w^{R}\left(y\right) \ge -\frac \theta 4 \quad \hbox{for any } y\in \overline{B_R(0)}.
\end{displaymath}
We deduce that, for any $\varepsilon<\varepsilon_0$,
\begin{displaymath}
  \alpha \phi^{\varepsilon, R} + H\left(\frac x \varepsilon, D \phi^{\varepsilon, R}  \right)\ge
 \frac \theta 2 \quad \hbox{in  } \overline {B_{\varepsilon R}(0)}.
\end{displaymath}

Next, using (\ref{eq:17}), consider a vanishing sequence $0<\varepsilon_n<\varepsilon_0$ such that  $u_ {\varepsilon_n}(0)$ tends to $\overline u (0)$.
We know that $u_ {\varepsilon_n}$ satisfies in the sense of viscosity
\begin{displaymath}
  \alpha \, u_ {\varepsilon_n} + H\left(\frac x {\varepsilon_n}, Du_ {\varepsilon_n} \right)\le
  0 \quad \hbox{in  } B_{\varepsilon_n R}(0).
\end{displaymath}
The comparison principle for Hamilton-Jacobi equations with state constraints,  see \cite{MR838056,MR951880} and \cite[ Th. 5.8, Chapter IV, page 278]{MR1484411}, then implies that
\begin{displaymath}
  \phi^{{\varepsilon_n}, R}- \frac  \theta {2 \alpha}\ge  u_ {\varepsilon_n}  \quad \hbox{in  } B_{\varepsilon_n R}(0),
\end{displaymath}
or, put differently,
\begin{displaymath}
   \overline u (0)+ {\varepsilon_n} w^R (\frac x {\varepsilon_n})- \frac  \theta {2 \alpha}\ge  u_ {\varepsilon_n}(x)  \quad \hbox{for  any } x\in  B_{\varepsilon_n R}(0).
\end{displaymath}
Taking $x=0$ and letting $n$ tend to $+\infty$ yields $  \overline u (0)- \frac  \theta {2 \alpha}\ge \overline u(0)$, the desired contradiction.
\end{proof}

The next proposition states that $\overline u$  is a viscosity subsolution of (\ref{eq:11}) (which is already known from Proposition~\ref{sec:homog-refeq:7-1}) and satisfies condition (\ref{eq:13}).

\begin{proposition}\label{sec:function--overline-4}
  The function $\overline u$ is a viscosity subsolution of $\alpha v+ \overline H (Dv )\le 0$ in the whole space $\R^d$.
\end{proposition}

\begin{proof}
  We know from  Proposition~\ref{sec:homog-refeq:7-1} that $\overline u$  is Lipschitz continous in $\R^d$. Hence, since $\overline u$ is a viscosity subsolution of $\alpha v+ \overline H (Dv )\le 0$ in $\R^d\setminus \{0\}$, it satisfies $\alpha \overline u(x)+ \overline H (D\overline u(x) )\le 0$ at almost every $x\in \R^d$, see \cite[Prop. 1.9, Chapter I, page 31]{MR1484411} and its proof. %connexit{\'e}
But $\overline H$ is convex. Therefore, from \cite[Prop. 5.1, Chapter II, page 77]{MR1484411}, $\overline u$ is a viscosity subsolution of $\alpha v+ \overline H (Dv )\le 0$ in the whole space $\R^d$.
\end{proof}

\subsection{The function $ \underline  u$ is a supersolution of  (\ref{eq:11}) and satisfies (\ref{eq:14})}\label{sec:function--underline}
We already know from Proposition~\ref{sec:homog-refeq:7-1}  that  $ \underline  u$ is a bounded supersolution of  (\ref{eq:11}).

Let $\phi\in C^1(\R^d)$ be such that $\underline u -\phi$ has a local minimum at the origin.  We wish to prove that
$ \alpha \underline u(0)+\max\left( E, \overline H(D\phi(0))\right)\ge 0$. The proof differs depending whether  $ \overline H(D\phi(0))>E$ or 
$\overline H(D\phi(0))\le E$. It is based on control theoretic arguments and  partly inspired from the ideas proposed by P-L Lions and P. Souganidis when they dealt with the case  $E= \overline H (p_0)$, see \cite{PLL-college}. However,  various new  arguments, reminiscent of those used in \cite{MR3299352}, will be needed to address the more difficult  case $E> \overline H (p_0)$.

Let us start by considering $p\in \R^d$ such that $ \overline H(p)>E$. It is obvious that $p\not =p_0$ because $E\ge \overline H (p_0)=\min_{q}\overline H (q)$.
From the convexity and the coercivity of $\overline H$,  we see that there exists a unique unit vector $e$ colinear to $p-p_0$ such that $\R\ni t\mapsto \overline H (p+te)$ is strictly decreasing in a neighborhood of $t=0$, and a unique vector $\tilde p$ such that
\begin{enumerate}
\item $\tilde p-p_0$ is colinear to $p-p_0$
\item $\overline H(\tilde p)=\overline H (p)$
\item $t\mapsto \overline H (\tilde p+te)$ is strictly increasing in a neighborhood of $t=0$.
\end{enumerate}
We easily check that $(p-p_0)\cdot e<0$ and that $(\tilde p-p_0)\cdot e>0$.

The following observation will be useful for proving that $\underline u$  satisfies (\ref{eq:14}).
\begin{remark}\label{sec:function--underline-3}
Note that if the function $x\mapsto  \underline u(x) - \underline u(0) - p\cdot x $ has a local minimum at $x=0$, then the function 
$x\mapsto  \underline u(x) - \underline u(0) - \min ( p\cdot x, \tilde p \cdot x)$ also has a local minimum at $x=0$. 
% Note that if the function $x\mapsto \underline u(0) + p\cdot x $ touches $\underline u$ from below at $x=0$, then the function 
% $x\mapsto  \underline u(0) + \min ( p\cdot x, \tilde p \cdot x)$ also touches $\underline u$ from below at $x=0$.
\end{remark}

We now assume that, for any $\lambda>0$ and $y\in \R^d$,  there exists an optimal trajectory $(z_\lambda, a_\lambda)$  of the optimal control problem 
\begin{displaymath}
 \chiperp^\lambda (y)=  \inf_{z,a} \left\{
 \int_0^\infty  e^{-\lambda t} \Bigl(  \ellper(z(t), a(t)) +   p\cdot \fper(z(t), a(t)) 
%- \lambda \langle \chiperp^\lambda\rangle
  \Bigr)
 dt 
  \right\},
\end{displaymath}
subject to
\begin{eqnarray}\label{eq:50}
  z'(t)&=& \fper(z(t),a(t)), \quad  a(t)\in A \; \hbox{ for almost } t>0,\\\label{eq:51}
  z(0)&=& y.
\end{eqnarray}
This assumption on the existence of such an optimal trajectory is made for simplicity. In the full generality, we may find trajectories as close to optimal as needed   
and what follows remains true with such trajectories.

The value function $ \chiperp^\lambda$  is a continuous function defined on $\R^d/\Z^d$; let $\langle \chiperp^\lambda \rangle$ stand for its mean value. It is well known, see \cite{LPV}, that  $-\lambda \langle \chiperp^\lambda \rangle $ tends to $\overline H (p)$ as $\lambda\to 0$ and that 
 after the extraction of a sequence,  $y\mapsto \chiperp^\lambda (y) -\langle \chiperp^\lambda \rangle $ tends to $\chiperp$ uniformly, where $\chiperp$ is a corrector associated with $p$ for the periodic homogenization problem. We recall that in periodic homogenization, the corrector  may not be unique, even up to the addition of a constant.

Similarly, let us assume that there exists $(\tilde z_\lambda, \tilde a_\lambda)$, an optimal trajectory for the optimal control problem 
\begin{displaymath}
 \chi_{{\rm per}, \tilde p}^\lambda (y)=  \inf_{z,a} \left\{
 \int_0^\infty  e^{-\lambda t} \left(  \ellper(z(t), a(t)) +   \tilde p\cdot \fper(z(t), a(t)) 
%- \lambda \langle \chiperp^\lambda\rangle  
\right)
 dt 
  \right\},
\end{displaymath}
subject to (\ref{eq:50})-(\ref{eq:51}).

The next lemma gives a uniform upper bound on the time spent by the optimal trajectories $z_{\lambda_n}$
in the half-space  $\{x\in \R^d:\;  e\cdot x \le r\}$ for $r>0$,  as $\lambda_n$ converges to $0$, and a symmetric estimate concerning
 $\tilde z_{\lambda_n}$. This lemma will be used in the proof of  Proposition \ref{sec:function--underline-1} below.

  \begin{lemma}
    \label{sec:function--underline-4}
If $\overline H (p)> E$,  let $(\lambda_n)_{n\in \N}$ and $(\tilde \lambda_n)_{n\in \N}$ be sequences converging to $0$ such that
 $y\mapsto \chiperp^{\lambda_n} (y) -\langle \chiperp^{\lambda_n} \rangle $ converges to $\chiperp$ uniformly and 
$y\mapsto \chi_{{\rm per},\tilde p}^{\tilde \lambda_n} (y) -\langle  \chi_{{\rm per},\tilde p}^{\tilde \lambda_n} \rangle $ converges to $\chi_{{\rm per},\tilde p}$ uniformly. There exists  a positive constant $c$  that only depends on $\overline H (p)$, such that for any $T>0$, there exist subsequences still denoted by $(\lambda_n)_{n\in \N}$ and $(\tilde \lambda_n)_{n\in \N}$
satisfying the following properties: for the unit vector $e$ introduced in the beginning of Section~\ref{sec:function--underline},
\begin{equation*}
 e\cdot (z_{\lambda_n}(t)-y)\ge c(t-1)\quad  \hbox{ and  } \quad e\cdot (\tilde z_{\lambda_n}(t)-y)\le -c(t-1),
\end{equation*}
 for all $n$  and $t\in [0,T]$.       
  \end{lemma}
  \begin{proof}
We will omit the index $n$, except at the end of the proof.
    We focus on the first assertion (concerning $z_\lambda$) since the proof of the second assertion (on $\tilde z_\lambda$) is similar.
Since $\overline H (p_0) \le E<\overline H (p)$, there exists $q = p_0 +\pi_0 (p-p_0)$, with $0<\pi_0 < 1$ such that $E< \overline H (q)<\overline H (p)$. The trajectory $(z_ \lambda,  a_\lambda)$ is then strictly suboptimal for the optimal control problem :
\begin{displaymath}
 \chi_{{\rm per}, q}^\lambda (y)=  \inf_{z,a} \left\{
 \int_0^\infty  e^{-\lambda t} \Bigl(  \ellper(z(t), a(t)) +   q\cdot \fper(z(t), a(t)) 
  \Bigr)
 dt 
  \right\},
\end{displaymath}
subject to (\ref{eq:50})-(\ref{eq:51}).
We obtain that 
\begin{eqnarray}
  \label{eq:chiperplambda}
  \chi_{{\rm per}, p}^\lambda (y) &=&  \ds \int_0 ^{t} e^{-\lambda s} \Bigl(   \ellper(z_ \lambda(s), a_ \lambda(s)) +   p\cdot \fper(z_ \lambda(s), a_ \lambda(s)
                                      \Bigr ) ds \\ \notag & & \ds +   e^{-\lambda t} \chi_{{\rm per}, p}^\lambda (z_\lambda(t)), \\
   \label{eq:chiperqlambda}
  \chi_{{\rm per}, q}^\lambda (y) & \le & \ds \int_0 ^{t} e^{-\lambda s} \Bigl(   \ellper(z_ \lambda(s), a_ \lambda(s)) +   q\cdot \fper(z_ \lambda(s), a_ \lambda(s)
 \Bigr )ds  \\ \notag  &  & \ds  +   e^{-\lambda t} \chi_{{\rm per}, q}^\lambda (z_\lambda(t)).
\end{eqnarray}
Using a Taylor expansion of $s\mapsto  e^{-\lambda s}$ as $\lambda$ vanishes, we see that there exists a positive constant $C$ such that 
\begin{eqnarray*}
  \left|\chi_{{\rm per}, p}^\lambda (y) -  \chi_{{\rm per}, p}^\lambda (z_\lambda(t)) -  \int_0 ^{t} \ellper(z_ \lambda(s), a_ \lambda(s)) ds - 
  p\cdot (z_ \lambda(t)-y) + \lambda t \chi_{{\rm per}, p}^\lambda (z_\lambda(t))\right| \\ \le C \lambda t^2
\end{eqnarray*}
and 
\begin{eqnarray*}
  \chi_{{\rm per}, q}^\lambda (y) -  \chi_{{\rm per}, q}^\lambda (z_\lambda(t)) -  \int_0 ^{t} \ellper(z_ \lambda(s), a_ \lambda(s)) ds - 
  q\cdot (z_ \lambda(t)-y) + \lambda t \chi_{{\rm per}, q}^\lambda (z_\lambda(t)) \\ \le C \lambda t^2
\end{eqnarray*}
We deduce from the latter two inequalities that
\begin{equation}\label{eq:52}
  \begin{split}
    &-\Bigl[\chi_{{\rm per}, q}^\lambda (z_\lambda) -    \langle \chi^\lambda_{{\rm per}, q} \rangle 
-  \chi_{{\rm per}, p}^\lambda (z_\lambda) + \langle \chi^\lambda_{{\rm per}, p} \rangle   \Bigr]_{0}^{t} 
\\
\le & (q-p)\cdot (z_\lambda(t)-y) -\lambda  t
(\chi_{{\rm per}, q}^\lambda (z_\lambda(t))-\chi_{{\rm per}, p}^\lambda (z_\lambda(t))) +2C\lambda t^2.
  \end{split}
\end{equation}
 We know that  for two correctors $\chi_{{\rm per}, p}$ and $\chi_{{\rm per}, q}$  of the periodic problems (\ref{eq:9}) 
respectively associated with $p$ and $q$,
 it is possible to  find some subsequence $\lambda$ such that
\begin{itemize}
\item  $\chi_{{\rm per}, q}^\lambda  -    \langle \chi^\lambda_{{\rm per}, q} \rangle $ converges uniformly to $\chi_{{\rm per}, q}$
\item  $\chi_{{\rm per}, p}^\lambda  -    \langle \chi^\lambda_{{\rm per}, p} \rangle $ converges uniformly to $\chi_{{\rm per}, p}$
\item $\lambda  (\chi_{{\rm per}, q}^\lambda -\chi_{{\rm per}, p}^\lambda )$ converges uniformly 
to  $-\overline H (q)+\overline H(p).$%\in (0, \overline H(p)) $.
\end{itemize}

Note that  $0<\overline H(p) -\overline H (q)<\overline H(p) $   and that 
$\chi_{{\rm per}, q}$ and $\chi_{{\rm per}, p}$ are uniformly bounded by a constant which depends on $\overline H(p)$.
Hence the left hand side of (\ref{eq:52}) is bounded from below by $-c_0$, for a suitable constant $c_0$ 
depending on $\overline H (p)$ only.

From (\ref{eq:52}) and the observations above, we deduce that there exists a constant $c_1$ depending on $\overline H (p)$ only, such that for all $T>0$, 
we may select a subsequence  $(\lambda_n) $  such that 
\begin{displaymath}
  (p-q)\cdot (z_{\lambda_n}(t)-y)\le (\overline H (q) - \overline H(p))t +c_1
\end{displaymath}
for all $t\in [0,T]$ and all $n$. The desired result follows.
  \end{proof}

% Let us introduce a smooth  function $\mathcal{V}: \R\to \R$ such that 
% \begin{enumerate}
% \item $\mathcal{V}(t)=0$ for $t\le -1$
% \item $\mathcal{V}(t)=1$ for $t\ge 1$
% \item $0\le \mathcal{V}(t)\le 1$ for $-1<t< 1$
% \end{enumerate}
% and define the smooth function  $v_{p,\tilde p}: \R^d \to \R$ by
% \begin{equation}
%   \label{eq:53}
% v_{p,\tilde p}(y)=  \tilde p\cdot y  + \mathcal{V}( e\cdot y ) (p-\tilde p)\cdot y.
% \end{equation}
% Note that 
% \begin{equation}
%   \label{eq:54}
% \varepsilon v_{p,\tilde p}\left(\frac x \varepsilon\right)=  \tilde p\cdot x + \mathcal{V}\left( e\cdot \frac x \varepsilon  \right) (p-\tilde p)\cdot x
% \end{equation}
%is a smooth approximation of $x\mapsto  \min ( p\cdot x, \tilde p \cdot x)$ which only depends on the variable $x\cdot e$.

The next proposition deals with the existence of a corrector associated to the piecewise linear function $y\mapsto \min(p\cdot y, \tilde p\cdot y)$, 
in the whole space $\R^d$. In view of Remark \ref{sec:function--underline-3}, this corrector will be useful for proving that $\underline u$ satisfies (\ref{eq:14}).
As already mentioned, this strategy differs from that used in \cite{PLL-college} in which  the  corrector  is associated to the  linear function $y\mapsto p\cdot y$, under stronger assumptions. 
\begin{proposition}
  \label{sec:function--underline-1}
For any $p\in \R^d$ such that $\overline H (p)>E$,  let $e$ and $\tilde p$
%, $v_{p,\tilde p} $
 be defined  as in the beginning of Section \ref{sec:function--underline}.
 Then there exists $\chi_{p,\tilde p}\in C(\R^d)$ such that
\begin{eqnarray}
  \label{eq:55}
\lim_{|y|\to \infty}  \frac {|\chi_{p,\tilde p}(y)-    \min ( p\cdot y, \tilde p \cdot y)  |}{|y|}=0,\\
 \label{eq:56}
 H\left(y, D\chi_{p,\tilde p } \right)=\overline H(p) \quad \hbox{in }\R^d,
\end{eqnarray}
where (\ref{eq:56}) is understood in the sense of viscosity.
\end{proposition}
\begin{proof}
Fix $\chiperp$ and $\chi_{{\rm per},\tilde p}$ as in Lemma~\ref{sec:function--underline-4}.
For a radius $R>R_0+1$ that will eventually converge to $+\infty$,  
consider the Dirichlet problem
\begin{eqnarray}
  \label{eq:57}
 H\left(y, D\chi^R_{p,\tilde p } 
%+ D v_{p,\tilde p}
\right)=\overline H(p) \quad \hbox{in }B_R(0),\\
\label{eq:58}
\chi^R_{p,\tilde p }(y)=\min \Bigl( p\cdot y + \chiperp(y), \tilde p\cdot y + \chi_{{\rm per}, \tilde p}(y)     \Bigr)  %-  v_{p,\tilde p}(y)  
 \quad \hbox{on }\partial B_R(0).
\end{eqnarray}

As a first step, let us construct a subsolution to (\ref{eq:57})-(\ref{eq:58}). For a constant $c>0$ that will be chosen below, we set 
\begin{equation}
  \label{eq:59}
\sigma(y)=\min \Bigl( w(y)-c,  p\cdot y + \chiperp(y),  \tilde p\cdot y + \chi_{{\rm per}, \tilde p}(y)    \Bigr),
\end{equation}
where $w$ is the viscosity solution of (\ref{eq:29}) constructed in Section~\ref{sec:ergod-const-glob}, which is Lipschitz continuous in $\R^d$ and thus satisfies  (\ref{eq:29}) almost everywhere in $\R^d$ from  \cite[Prop. 1.9, Chapter I, page 31]{MR1484411}.

We next choose $c$ such that, for any  $y\in B_{R_0+1}(0)$,
\begin{displaymath}
  w(y)-c  <   \min\Bigl(p\cdot y + \chiperp(y),  \tilde p\cdot y + \chi_{{\rm per}, \tilde p}(y)    \Bigr) , 
\end{displaymath}
thus 
\begin{equation}
  \label{eq:60}
  \sigma(y)=  w(y)-c, \quad \hbox{for any }y \in B_{R_0+1}(0).
\end{equation}
On the other hand, since $E<\overline H (p)$,  $w$, $y\mapsto p\cdot y+\chiperp(y)$ and $y\mapsto \tilde p\cdot y+ \chi_{{\rm per}, \tilde p}(y)$ are three 
Lipschitz continuous viscosity subsolutions of $H_{{\rm per}}(y, Dv)\le \overline H (p)$ in $ B_R(0)\setminus\overline{B_{R_0+1}(0)}$.
 We deduce from \cite[Prop. 1.9, Chapter I, page 31]{MR1484411} that 
\begin{equation}\label{eq:61}
   H\left(y, D\sigma (y) \right)\le \overline H(p) \quad \hbox{ for almost all } y\in B_R(0)\setminus \overline{ B_{R_0+1}(0)}.
\end{equation}
Combining this with (\ref{eq:60}), $\sigma $ actually satisfies (\ref{eq:61}) for almost all $ y\in B_R(0)$.
 Then, since $H$ is convex with respect to its second argument, 
\cite[Prop. 5.1, Chapter II, page 77]{MR1484411} can be applied and yields that $\sigma$ is a viscosity subsolution of 
\begin{equation}\label{eq:62}
   H\left(y, D\sigma  \right)\le \overline H(p) \quad \hbox{ in } B_R(0).
\end{equation}
On the other hand,  it is clear that 
\begin{displaymath}
  \sigma(y)\le  \min\Bigl(p\cdot y + \chiperp(y),  \tilde p\cdot y + \chi_{{\rm per}, \tilde p}(y)    \Bigr)  \quad \hbox{ for any } y\in \partial B_R(0).
\end{displaymath}
Hence, $\sigma$ is a subsolution of the Dirichlet problem (\ref{eq:57})-(\ref{eq:58}).

%Since $\overline H (p)> E\ge \overline H (p_0)= \min_q \overline H (q)$, we may apply Proposition \ref{sec:ergod-const-glob-2}. 

Our second step is to establish the existence of a constant 
 $C>0$ independent of $R>R_0+1$ such that, for all $y\in B_R(0)$,
 \begin{equation}
   \label{eq:63}
 \sigma(y) -  \min\Bigl(p\cdot y + \chiperp(y),  \tilde p\cdot y + \chi_{{\rm per}, \tilde p}(y)    \Bigr)\ge -C.
 \end{equation}
%This estimate will be an ingredient in order to prove (\ref{eq:55}).
For this purpose, we distinguish two cases:
\begin{description}
\item[Case 1: $E> \overline H (p_0)$.]
We first observe that  $y\mapsto \min\Bigl((p-p_0)\cdot y,  (\tilde p-p_0)\cdot y    \Bigr)$ is bounded from above.
Therefore,  $y\mapsto \min\Bigl((p-p_0)\cdot y  + \chiperp(y),  (\tilde p-p_0)\cdot y  + \chi_{{\rm per}, \tilde p}(y)   \Bigr)$ is also bounded from above. 
Since $\overline H (p)> E> \overline H (p_0)$, we deduce from this and  (\ref{eq:31}) in Proposition \ref{sec:ergod-const-glob-2}
  that for $|y|$ large enough, $w(y)-p_0\cdot y -c$ is larger that $ \min\Bigl((p-p_0)\cdot y  + \chiperp(y),  (\tilde p-p_0)\cdot y  + \chi_{{\rm per}, \tilde p}(y)   \Bigr)$, hence
\begin{displaymath}
  \sigma(y)=  \min\Bigl(p\cdot y + \chiperp(y),  \tilde p\cdot y + \chi_{{\rm per}, \tilde p}(y)    \Bigr).
\end{displaymath}
This yields (\ref{eq:63}).
\item[Case 2: $E=  \overline H (p_0)$.] From the second conclusion of  Proposition \ref{sec:ergod-const-glob-2}, we know that there exists a 
a constant $C>0$ independent of $R$ such that, for all $y\in \R^d$,  
\begin{equation}
  \label{eq:64}
 w(y)-p_0\cdot y>-C.
\end{equation}
 %For a constant $C>0$ independent of $R$, ($C$ will change from one line to the next in what follows).
 Hence, (\ref{eq:60}), which also reads
\begin{equation}
  \label{eq:65}
w(y)-p_0\cdot y-c< \min\Bigl((p-p_0)\cdot y + \chiperp(y),  (\tilde p-p_0)\cdot y + \chi_{{\rm per}, \tilde p}(y)    \Bigr)  
\end{equation}
then yields that $|e\cdot y|$ is bounded by a constant uniform in $R$.
\\
In turn, this implies that
\begin{displaymath}
   \min\Bigl( (p-p_0)\cdot y + \chiperp(y),  (\tilde p-p_0)\cdot y + \chi_{{\rm per}, \tilde p}(y)    \Bigr) 
\end{displaymath}
is bounded above uniformly in $R$. Combining this with   \eqref{eq:64}, we deduce that for a possibly different value of the  constant $C$,
%if (\ref{eq:65}) is true, then there also holds 
\begin{displaymath}
  w(y)-c -  \min\Bigl(p\cdot y + \chiperp(y),  \tilde p\cdot y + \chi_{{\rm per}, \tilde p}(y)    \Bigr)\ge -C.
\end{displaymath}
Inserting this minoration into \eqref{eq:59}  yields (\ref{eq:63}).
\end{description}

%  holds. Hence, for $R$ large enough,
% \begin{displaymath}
%   \sigma(y)=  \min\Bigl(p\cdot y + \chiperp(y),  \tilde p\cdot y + \chi_{{\rm per}, \tilde p}(y)    \Bigr)  \quad \hbox{ for any } y\in \partial B_R(0).
% \end{displaymath}
% \item $E=\overline H (0)$:
% \end{enumerate}

\medskip

Our third step consists of finding a supersolution to  (\ref{eq:57})-(\ref{eq:58}). This is easy, because  for $q\in \R^d$,  $|q|$ large enough  and  $y_0\in \R^d$  such that $-q\cdot y_0$ is large enough,  $y\mapsto q\cdot (y-y_0)$  is indeed a supersolution of (\ref{eq:57})-(\ref{eq:58}). 

\medskip

The existence  of a solution $\chi^R_{p,\tilde p}$ to  (\ref{eq:57})-(\ref{eq:58}) is then obtained by Perron's method. It can be proved in a classical way that this solution is unique.  Besides, we have the following  representation formula:
\begin{equation}
  \label{eq:66}  \chi^R_{p,\tilde p}(y)=\inf_{z,a, \tau} \left\{
    \begin{array}[c]{l}\ds
\int_0^\tau  \Bigl(\ell(z(t), a(t))+\overline H (p)\Bigr) dt \\  \ds+    \min\Bigl(p\cdot z(\tau) + \chiperp(z(\tau)),
  \tilde p\cdot z(\tau) + \chi_{{\rm per}, \tilde p}(z(\tau))          \Bigr)
    \end{array}
  \right\},
\end{equation}
subject to
\begin{eqnarray*}
  z'(t)&=& f(z(t),a(t)), \quad a(t)\in A \; \hbox{ for almost } t>0,\\
  z(0)&=& y,\\
\tau&=& \inf \{t>0:  |z(t)|=R \}.
\end{eqnarray*}

Now, our goal is to  obtain $\chi_{p,\tilde p}$ as a locally uniform limit of  $\chi^{R_n}_{p,\tilde p}$ for some sequence $R_n$ tending to $+\infty$.
For that purpose, we will use (\ref{eq:63}) which readily gives a bound from below on $\chi^{R}_{p,\tilde p}$,
 but we also need an accurate bound from above.  Our fourth step therefore consists of obtaining this bound by using the representation formula~(\ref{eq:66}) and Lemma \ref{sec:function--underline-4}.

With $c$ given in Lemma \ref{sec:function--underline-4}, which depends only on $\overline H (p)$, let $T>0$ be such that $c(T-1)>2R$.
We denote by $\lambda$  any term in the sequence  appearing in Lemma \ref{sec:function--underline-4}.
With $(z_\lambda, a_\lambda)$ introduced above, let $\tau_\lambda$ be the first time at which $z_\lambda$ hits $\partial B_R(0)$,
which is smaller than $T$ thanks to Lemma~\ref{sec:function--underline-4}. From (\ref{eq:66}), we deduce that 
\begin{equation*}
  \begin{split}
  \chi^R_{p,\tilde p}(y)  \le & \left(   \begin{array}[c]{l}\ds
\int_0^{\tau_\lambda} \left(\ell(z_\lambda(t), a_\lambda(t))+\overline H (p)\right) dt \\  \ds+    \min\Bigl(p\cdot z_\lambda(\tau_\lambda) + \chiperp(z_\lambda(\tau_\lambda)),
  \tilde p\cdot z_\lambda(\tau_\lambda) + \chi_{{\rm per}, \tilde p}(z_\lambda(\tau_\lambda))          \Bigr)
    \end{array}\right)
    \\ 
 \le &    \ds
\int_0^{\tau_\lambda} \left(\ell(z_\lambda(t), a_\lambda(t))+\overline H (p) \right) dt  +  
  p\cdot z_\lambda(\tau_\lambda) + \chiperp(z_\lambda(\tau_\lambda))\\
  = &    \ds
  A +  
  p\cdot y
% \int_0^{\tau_\lambda} \left(\ell(z_\lambda(t), a_\lambda(t))+    p\cdot \fper( z_\lambda(t), a_\lambda(t))+  \overline H (p) \right)dt  + \chiperp(z_\lambda(\tau_\lambda)).
  \end{split}
\end{equation*}
where,  for brevity, we have set
\begin{displaymath}
  A= \ds  \int_0^{\tau_\lambda} \Bigl(\ell(z_\lambda(t), a_\lambda(t))+    p\cdot \fper( z_\lambda(t), a_\lambda(t))+  \overline H (p) \Bigr)dt + \chiperp(z_\lambda(\tau_\lambda)).
\end{displaymath}
% In what follows, the constant $c_1$ depends  only on
% $\overline H (p)$ and  may vary from line to line.
Observing that,   from Lemma \ref{sec:function--underline-4}, the measure of $\{t:\;  z^\lambda(t)\in B_{R_0}(0) \}$ is bounded,
  and that $\ell$ coincides with  $\ellper $ outside $B_{R_0}(0)$, we first obtain that,  for a positive constant $c_1$ depending  only on $\overline H (p)$,
\begin{eqnarray*}
   A  \le &  \ds  \int_0^{\tau_\lambda} \Bigl(\ellper(z_\lambda(t), a_\lambda(t))+    p\cdot \fper( z_\lambda(t), a_\lambda(t))+  \overline H (p) \Bigr)dt % \\&
    +  
    \chiperp(z_\lambda(\tau_\lambda)) + c_1\\
   =&   \ds \int_0^{\tau_\lambda}  e^{-\lambda t} \Bigl(\ellper(z_\lambda(t), a_\lambda(t))+    p\cdot \fper( z_\lambda(t), a_\lambda(t)) \Bigr)dt   +   e^ {-\lambda \tau_\lambda}\chiperp(z_\lambda(\tau_\lambda)) \\
   &+ \tau_\lambda \overline H (p)  + c_1 + B,
\end{eqnarray*}
% \begin{equation*}
%   \begin{array}[c]{lcl}
% % &   \ds  \int_0^{\tau_\lambda} \Bigl(\ell(z_\lambda(t), a_\lambda(t))+    p\cdot \fper( z_\lambda(t), a_\lambda(t))+  \overline H (p) \Bigr)dt + \chiperp(z_\lambda(\tau_\lambda))\\
%    A  &\le &  \ds  \int_0^{\tau_\lambda} \Bigl(\ellper(z_\lambda(t), a_\lambda(t))+    p\cdot \fper( z_\lambda(t), a_\lambda(t))+  \overline H (p) \Bigr)dt % \\&
%     +  
%     \chiperp(z_\lambda(\tau_\lambda)) + c_1\\
%    &=&   \ds \int_0^{\tau_\lambda}  e^{-\lambda t} \Bigl(\ellper(z_\lambda(t), a_\lambda(t))+    p\cdot \fper( z_\lambda(t), a_\lambda(t)) \Bigr)dt   +   e^ {-\lambda \tau_\lambda}\chiperp(z_\lambda(\tau_\lambda)) \\
%   & &+ \tau_\lambda \overline H (p)  + c_1 + B,
%   \end{array}
% \end{equation*}
% Next, we can write the right hand side of the latter inequality  as follows:
% \begin{displaymath}
%   \begin{array}[c]{l}
%  \ds  \int_0^{\tau_\lambda}  e^{-\lambda t} \Bigl(\ellper(z_\lambda(t), a_\lambda(t))+    p\cdot \fper( z_\lambda(t), a_\lambda(t)) \Bigr)dt   +   e^ {-\lambda \tau_\lambda}\chiperp(z_\lambda(\tau_\lambda)) \\
%   + \tau_\lambda \overline H (p)  + c_1 + B,
%   \end{array}
% \end{displaymath}
where
\begin{displaymath}
  \begin{array}[c]{ll}
    B=&  (1-e^ {-\lambda \tau_\lambda})\chiperp(z_\lambda(\tau_\lambda)) \\
    &+  \ds  \int_0^{\tau_\lambda}  (1-e^{-\lambda t}) \Bigl(\ellper(z_\lambda(t), a_\lambda(t))+    p\cdot \fper( z_\lambda(t), a_\lambda(t)) \Bigr)dt.
  \end{array}
\end{displaymath}
But $\tau_\lambda$ is bounded uniformly in $\lambda$ (it does depend on $R$), so
$\lim_{\lambda\to 0}  (1-e^ {-\lambda \tau_\lambda})=0$,  and we see that
for $\lambda $ smaller than a constant depending on $R$, $|B|\le c_1$, hence
\begin{eqnarray*}
   A \le 
 &\ds  \int_0^{\tau_\lambda}  e^{-\lambda t} \Bigl(\ellper(z_\lambda(t), a_\lambda(t))+    p\cdot \fper( z_\lambda(t), a_\lambda(t)) \Bigr)dt +   e^ {-\lambda \tau_\lambda}\chiperp(z_\lambda(\tau_\lambda)) \\
 &+ \tau_\lambda \overline H (p)  + 2c_1
\end{eqnarray*}
Using \eqref{eq:chiperplambda}, the right hand side in the latter inequality can be written
\begin{equation*}
  \chiperp^\lambda(y)- e^{-\lambda  \tau_\lambda} \chiperp^\lambda(z_\lambda(\tau_\lambda)) +  e^ {-\lambda \tau_\lambda}\chiperp(z_\lambda(\tau_\lambda)) 
 + \tau_\lambda \overline H (p)  + 2c_1.
\end{equation*}
Next, since $\lim_{\lambda\to 0}  \lambda \langle \chiperp^\lambda \rangle =-\overline H (p)$, we get that for $\lambda$ small enough,
\begin{eqnarray*}
 A \quad \le 
    &\ds \quad \chiperp^\lambda(y)-
      e^{-\lambda  \tau_\lambda} \chiperp^\lambda(z_\lambda(\tau_\lambda)) +
      e^ {-\lambda \tau_\lambda}\chiperp(z_\lambda(\tau_\lambda)) 
      -\lambda  \tau_\lambda   \langle \chiperp^\lambda \rangle
      + 3c_1
    \\
    = & \ds \chiperp^\lambda(y)-  \langle \chiperp^\lambda \rangle 
    +\left( -e^{-\lambda  \tau_\lambda} +1-\lambda  \tau_\lambda\right)  \langle \chiperp^\lambda \rangle \\
    &  \ds - e^{-\lambda  \tau_\lambda} \left(\chiperp^\lambda(z_\lambda(\tau_\lambda)) -  \langle \chiperp^\lambda \rangle - \chiperp(z_\lambda(\tau_\lambda))  \right)+3c_1 
\end{eqnarray*}
The uniform convergence of $ \chiperp^\lambda -\langle \chiperp^\lambda \rangle $  to $\chiperp$ then implies that for $\lambda$ small enough
\begin{displaymath}
  A\le  \chiperp(y) +4 c_1.
\end{displaymath}
We have proven that for all $R$,
\begin{equation}\label{eq:67}
  \chi^R_{p,\tilde p}(y)\le   \chiperp(y)  + p\cdot y +4c_1.
\end{equation}
Similarly, using  Lemma~\ref{sec:function--underline-4}, we see that there exists a positive constant $\tilde c_1$ such that
\begin{equation}\label{eq:68}
  \chi^R_{p,\tilde p}(y)\le   \chi_{{\rm per},\tilde p}(y)  + \tilde p\cdot y +4 \tilde c_1.
\end{equation}
This concludes our fourth step. We deduce from (\ref{eq:67}) and (\ref{eq:68}) that there exists a constant $c_2$ independent of $R$, (which depends only on $\overline H (p)$) such that for all $y\in \overline B_R(0)$, 
\begin{equation}\label{eq:69}
  \sigma (y) \le \chi^R_{p,\tilde p}(y)\le   \min \left(   \chiperp (y) +p\cdot y, \chi_{{\rm per},\tilde p}(y)  + \tilde p\cdot y \right) +c_2.
\end{equation}
From (\ref{eq:69}) and the fact that $\chi^R_{p,\tilde p}$ is locally Lipschitz with a Lipschitz constant independent of $R$, we can 
find a sequence $(R_n)_n$ which tends to $+\infty$ such that $\chi^{R_n}_{p,\tilde p}$ tends to some $\chi_{p,\tilde p}$ locally uniformly, and 
by passing to the limit in (\ref{eq:57}), (\ref{eq:69}) (using  (\ref{eq:63})), we obtain that
$\chi_{p,\tilde p}$
satisfies (\ref{eq:55})-(\ref{eq:56}).
\end{proof}

% \begin{remark} \label{sec:function--underline-1-rem}
%   Proposition  \ref{sec:function--underline-1} states the existence and uniqueness of a corrector associated to the piecewise linear function $y\mapsto  \min ( p\cdot y, \tilde p \cdot y) $.  The proof uses arguments from the theory of optimal control. By contrast with
% \cite{PLL-college} in which stronger assumptions were made, we have not been able to construct a corrector  associated to the simpler linear function $y\mapsto p\cdot y$ with the same kind of arguments, because it seems more difficult to  estimate  the time spent by $z_\lambda$ near the defect uniformly with respect to the starting point $y$.
% \end{remark}

\begin{proposition}
  \label{sec:function--underline-2}
If there exists $\phi\in C^1(\R^d)$  such that $0$ is a local minimizer of $\underline u -\phi$  and $ \overline H(D\phi(0))>E$, then
\begin{equation}
  \label{eq:70}
 \alpha \underline u(0)+ \overline H(D\phi(0)) \ge 0.
\end{equation}
\end{proposition}

\begin{proof}
We can always assume that $\phi(0)=\underline u(0)$ and that $\underline u -\phi$ has a strict local minimum at the origin.
For brevity, let us set $p = D\phi(0)$. Because $ \overline H(p)>E$, we know that $p\not=p_0$  and we can  apply 
Proposition~\ref{sec:function--underline-1},  with $e$ and $\tilde p$ defined as above. Suppose by contradiction that
\begin{equation}
  \label{eq:71}
  \alpha \underline u(0)+ \overline H(p) =-\theta<0,
\end{equation} and
 consider the perturbed test-function
  \begin{displaymath}
    \phi_\varepsilon(x)=\phi(0)+\varepsilon \chi_{p,\tilde p}(\frac x \varepsilon),
  \end{displaymath}
where $\chi_{p,\tilde p}$ is the  function appearing in Proposition~\ref{sec:function--underline-1}.

%From~(\ref{eq:56}), the regularity of $\phi$ and (\ref{eq:4}), and from (\ref{eq:55}), 
The definition of $\phi_\varepsilon$, \eqref{eq:56} and \eqref{eq:71} imply that
 $\alpha \phi_\varepsilon +H(\frac x \varepsilon , D \phi_\varepsilon)\le \alpha\varepsilon\chi_{p,\tilde p}(\frac x \varepsilon) -\theta$ in the sense of viscosity in $\R^d$.

From this and \eqref{eq:55}, we deduce that there exists $r_0>0$ and $\varepsilon_0>0$ such that for all $0<\varepsilon<\varepsilon_0 $ and 
$0<r<r_0$,  $ \phi_\varepsilon$ is a viscosity subsolution of 
\begin{equation}
  \label{eq:72}
  \alpha \phi_\varepsilon+ H(\frac x  \varepsilon , D\phi_\varepsilon )\le -\frac \theta 2 \quad \hbox{ in } B_r(0).
\end{equation}
On the other hand, since $0$ is a strict local minimizer of $\underline u -\phi$, there exists $r_1>0$ and a  function   $k: (0,r_1]\to (0,1]$, 
such that $\lim_{r\to 0} k(r)=0$ and  for any $r\in (0,r_1]$,
\begin{displaymath}
  \phi(x)\le \underline u (x ) -k(r) \quad \hbox{ on } \partial B_r(0).
\end{displaymath}
From \eqref{eq:55}, we know that for $x\not =0$, $\varepsilon \chi_{p,\tilde p} (\frac x \varepsilon)= \min(p\cdot x, \tilde p \cdot x)+ |x| o_{\varepsilon \to 0}(1)$.

Using (\ref{eq:18}), this implies that  first fixing $r>0$ small enough, we have for  $\varepsilon$ small enough,
% From Lemma \ref{sec:function--underline-1},
% \begin{equation*}
%   \begin{split}
% \phi_\varepsilon(x)&=\phi(x)+O(\varepsilon)  + |x|   \frac {\chi_p(\frac x \varepsilon)}{ \frac {|x|} \varepsilon} \\ &=\phi(x)+O(\varepsilon)  +   {o(1)} _{\varepsilon \to 0}|x|.
%   \end{split}
% \end{equation*}
% We deduce that
% for $\varepsilon$ and $r$ small enough, 
\begin{equation}
  \label{eq:73}
\phi_\varepsilon(x)< u_\varepsilon (x) -\frac {k(r)}  2   \quad  \hbox{ on } \partial B_r(0).
\end{equation}
From (\ref{eq:72}) and (\ref{eq:73}) and since $u_\varepsilon$ is a viscosity solution of (\ref{eq:7}), the comparison principle yields  that it is possible to choose $r>0$ such that, for $\varepsilon$ small enough, $\phi_\varepsilon(x)\le  u_\varepsilon (x) -\frac {k(r)}  2$ in $\overline  {B_r(0)}$.
By choosing a sequence $\varepsilon_n$ such that $ u_{\varepsilon_n} (0)$ tends to $\underline u (0)$, we deduce that
%$\phi_{\varepsilon_n}(0) +\frac {k(r)}  2\le  \underline u(0) $ and finally that 
$\phi(0)+\frac {k(r)}  2\le  \underline u(0) $, the desired contradiction.
\end{proof}

\begin{proposition}
\label{sec:function--underline-5}
%Assume that $E> \min_{q}\overline H (q)$.
If there exists $\phi\in C^1(\R^d)$  such that $0$ is a local minimizer of $\underline u -\phi$  and $ \overline H(D\phi(0))\le E$, then
\begin{equation}
\label{eq:74}
 \alpha \underline u(0)+ E \ge 0.
\end{equation}
\end{proposition}
\begin{proof}
As above we may assume without loss of generality that $\phi(0)=\underline u (0)$ and that $0$ is a strict local minimizer of  $\underline u -\phi$.
  Let us again set $p= D\phi(0)$.

 If $ \overline H(p)> \overline H(p_0)= \min_{q}\overline H (q)$, we set $e$ and $\tilde p$ as above.
For $\eta, \tilde \eta>0$, the function $x\mapsto  \underline u(x) - \underline u(0)- \min ( (p_0+ (1+\eta)(p-p_0) )\cdot x, (p_0+ (1+\tilde \eta) (\tilde p-p_0))\cdot x)$ has a local minimum at at $x=0$.
%touches $\underline u$ from below at $x=0$.
Given $\varepsilon>0$, we may choose $\eta$ and $\tilde \eta$ such that $  \overline H(p_0+ (1+\eta)(p-p_0) )
=  \overline H(p_0+ (1+\tilde \eta) (\tilde p-p_0)) =E+\varepsilon$.
The same argument as in the proof of Proposition \ref{sec:function--underline-2} with
$\chi_{p_0+ (1+\eta)(p-p_0), p_0+ (1+\tilde \eta) (\tilde p-p_0)}$, the corrected version of the piecewise affine function $y\mapsto \min ( (p_0+ (1+\eta)(p-p_0) )\cdot y, (p_0+ (1+\tilde \eta) (\tilde p-p_0))\cdot y)$,  yields 
$ \alpha \underline u(0)+  \overline H(p_0+ (1+\tilde \eta) (\tilde p-p_0)) \ge 0$, that is
\begin{displaymath}
  \alpha \underline u(0)+E+\varepsilon \ge 0.
\end{displaymath}
Letting $\varepsilon$ tend to $0$, we obtain (\ref{eq:74}).

On the other hand, if $\overline H(p)= \min_{q}\overline H (q)=\overline H (p_0)$, we then consider  two cases:
\begin{enumerate}
\item If $p\not=p_0$, then from the convexity of $=\overline H$, the set $\{t\in \R ,  \overline H(p_0 +t (p-p_0))=\overline H(p_0)\}$ is an interval  $ [ \beta, \gamma]$, with $\beta\le 0$ and $\gamma\ge 1$.
Let  $e$ be the unique unit vector aligned with $p-p_0$ such that $t\mapsto \overline H (   p_0 + \gamma (p-p_0) +t e)$ is non increasing near $t=0$. 
For $\eta, \tilde \eta>0$, the function $x\mapsto \underline u(0)+ \min ( ( p_0+ ( \gamma+\eta)(p-p_0)) \cdot x,
(p_0 +(\beta -\tilde \eta) ) (p-p_0)\cdot x)$ touches $\underline u$ from below at $x=0$. Given $\varepsilon>0$,
we may choose $\eta$ and $\tilde \eta$  such that $  \overline H(  p_0+ ( \gamma+\eta)(p-p_0))= \overline H(p_0 +(\beta -\tilde \eta)  (p-p_0))=E+\varepsilon$.
The same argument as in the proof of Proposition \ref{sec:function--underline-2} with

$\chi_  { p_0+ ( \gamma+\eta)(p-p_0) , p_0 +(\beta -\tilde \eta)  (p-p_0)  }$,
the corrected version of the piecewise affine function $y\mapsto  \min ( ( p_0+ ( \gamma+\eta)(p-p_0)) \cdot y,
(p_0 +(\beta -\tilde \eta) ) (p-p_0)\cdot y)$,
yields that 
$  \alpha \underline u(0)+E+\varepsilon \ge 0$. Then (\ref{eq:74}) is obtained by letting $\varepsilon$ tend to $0$.
 \item If $p=p_0$, we may choose any non zero vector $q_\varepsilon$ such that $\overline H(q_\varepsilon)=E+\varepsilon$: let  $e$ be the unique unit vector aligned with $q_\varepsilon-p_0$ such that $t\mapsto \overline H (q_\varepsilon +t e)$ is non increasing near $t=0$ and let $\tilde q_\varepsilon$ be the unique vector different from $q_\varepsilon$  such that $\tilde q_\varepsilon-p_0 $ aligned with $q_\varepsilon-p_0$, and $\overline H(\tilde q_\varepsilon)=E+\varepsilon$.
The  function $x\mapsto \underline u(0)+ \min (  q_\varepsilon\cdot x,  \tilde  q_\varepsilon\cdot x)$ touches $\underline u$ from below at $x=0$, which yields 
 $ \alpha \underline u(0)+E+\varepsilon \ge 0$ as in the proof of Proposition  \ref{sec:function--underline-2},
 using the   corrector $\chi_  {q_\varepsilon, \tilde q_\varepsilon  }$. Letting $\varepsilon$ tend to $0$, we obtain (\ref{eq:74}).
\end{enumerate}

\end{proof}

We have proved that $\underline u$ satisfies (\ref{eq:14}). 
In order to prove Theorem~\ref{sec:main-result-1}, there only remains  to establish that   $\underline u=\overline u$.

\subsection{End of the proof of Theorem~\ref{sec:main-result-1}}\label{proof_of_main_th}
We consider two cases:
\paragraph{Case 1}
 Suppose first that there does not exist any function $\phi\in C^1(\R^d)$
such that $\underline u -\phi$ has a local minimum at the origin and that $\overline H (D\phi(0))\le E$. 
Hence, from Proposition \ref{sec:function--underline-2}, for any function 
$\phi\in C^1(\R^d)$  such that $\underline u -\phi$ has a local minimum at the origin, 
  \begin{displaymath}
    \alpha \underline u(0)+ \overline H(D\phi(0))\ge 0.
  \end{displaymath}
We deduce from this and Lemma \ref{sec:homog-refeq:7-1} that $\underline u$ is a bounded supersolution of (\ref{eq:8}) in $\R^d$. 
On the other hand, from  Proposition \ref{sec:function--overline-4},  $\overline u$ is a bounded subsolution of (\ref{eq:8}) in $\R^d$.
Hence, from the comparison principle,  \cite[Th. 2.12, Chapter III, page 107]{MR1484411},  we deduce that  for any $x\in \R ^d$, $\overline u(x)\le \underline u(x)$. From the definitions (\ref{eq:17})-(\ref{eq:18}) of $\overline u$  and $\underline u$, we deduce that $\underline u=\overline u$. 
\paragraph{Case 2} If, on the contrary,  there   exists $\phi\in C^1(\R^d)$  such that $\underline u -\phi$ has a local minimum at the origin
 and  that $\overline H (D\phi(0))\le E$, Proposition~\ref{sec:function--underline-5} yields that $\alpha \underline u(0)+E\ge 0$.
On the other hand, from Proposition \ref{sec:ergod-const-glob-1}, we know that $\alpha \overline u(0)+E\le 0$.  From the definitions  (\ref{eq:17})-(\ref{eq:18}) of $\overline u$  and $\underline u$, we deduce that $\underline u(0)=\overline u(0)$.  This allows us to apply the comparison principle in $\R^d\setminus \{0\}$,
 \cite[Remark 2.14, Chapter III, page 109]{MR1484411},
 because  $\underline u$  and $\overline u$ are respectively  a bounded supersolution 
and a bounded subsolution of (\ref{eq:11}) in $\R^d\setminus\{0\}$.
 Therefore,  for any $x\in \R ^d$, $\overline u(x)\le \underline u(x)$, and finally, $\overline u= \underline u$.

\medskip

In both cases, we deduce that the whole family $u_\varepsilon$ converges locally uniformly to a solution of  (\ref{eq:11}) through (\ref{eq:14}), 
which concludes the proof of Theorem~\ref{sec:main-result-1}.

\section{Miscellaneous comments and extensions}
\label{sec:extension}
This section contains various comments about, and illustrations of the general result proven above.

% section 1D
\subsection{A simple and explicit one-dimensional example}
\label{sec:1D}

We begin by providing in this Section~\ref{sec:1D} some illustrations of the results in a simple one-dimensional setting. 

Upon considering this simple situation, we expect to illustrate as sharply as possible,  the homogenization process  described by Theorem~\ref{sec:main-result-1}. We also intend to (a) make explicit in those simple settings the key quantities involved in Theorem~\ref{sec:main-result-1}, in particular the constant~$E$ and the Dirichlet condition~\eqref{eq:12}-\eqref{eq:13}-\eqref{eq:14}
and (b) provide some additional light on some qualitative aspects of the problem that are best exposed using simple examples
and that will be generalized in higher dimensions later on.

We will exploit the peculiarity of the considered one-dimensional setting. In the course of our arguments, we will make several simplifying assumptions about the type of  Hamiltonian~$H$ and the shape of the defect(s) that we consider: see for instance~ \eqref{eq:l0-support}, \eqref{eq:HJB-3}, 
%\eqref{eq:support-a},
\eqref{eq:even}
% \eqref{eq:random-negatif-et-etroit}
below, etc. Even if we expect that several of our arguments may carry over when some of these simplifying assumptions are relaxed, it might be the case that not everything can be generalized. We do not claim that the \emph{techniques} we are employing below, which are specific to the particular setting, carry over to higher dimensions and other Hamiltonians. We only intend to illustrate, on some simple enough cases, some \emph{phenomena} that we find interesting, with no sake of generality whatsoever. 

\medskip

We consider in this section what is, to some extent, the \emph{simplest possible setting} where the homogenization process we have established in the general setting occurs. We pose the Hamilton-Jacobi equation on the real line~$\R$. We pick a \emph{separate} Hamiltonian~$H$ (by this term, we mean, throughout this section, an Hamiltonian that depends \emph{separately} of~$x$ and~$p$), with a potential part that reads as the sum 
\begin{equation}
\label{eq:addition}
\ell\,=\,\ellper\,+\,\ell_0,
\end{equation}
of a periodic function~$\ellper$  and a  local defect~$\ell_0$. Both $\ellper$ and~$\ell_0$ are assumed smooth.
We also assume that~$\ell_0$ is compactly supported, with, say,
\begin{equation}
\label{eq:l0-support}
\textrm{Supp}\,\ell_0\subset [-\frac12,\frac12].
\end{equation}
In order to observe an actual perturbation by~$\ell_0$ of the homogenized equation, we further assume 
\begin{equation}
\label{eq:l0-downward}
\inf_{\R}(\ellper\,+\,\ell_0)\,<\,\inf_{\R}\ellper\,.
\end{equation}
The role of this condition will become clear below but it is already intuitive that its spirit is to make sure that the defect indeed locally \emph{lower} the periodic environment (or put differently in the optimal control interpretation, diminishes the periodic cost) and therefore indeed shows up in the homogenized limit. Condition~\eqref{eq:l0-downward} is our only \emph{actual} assumption, all the other assumptions above and below  (smoothness of the data, compact support of~$\ell_0$, and others to come) being only used to simplify the algebraic expressions manipulated and to spare the reader unnecessary technicalities.
  As for the kinetic part, we choose~$|p|$, that is a simple, commonly used convex nonlinearity. In our notation of Section~\ref{sec:setting}, our choices of course correspond to~$A\,=\,[-1,1]$ and $f(x,a)\,=\,\fper(x,a)\,=\,a$. Put differently, there is neither oscillation nor defect in this kinetic part. We thus manipulate
\begin{equation}
\label{eq:separated}
H(x,p)\,=\,|p|\,-\,\ellper(x)\,-\,\ell_0(x)
\end{equation}
throughout this Section~\ref{sec:1D}. 

Because this positive constant is irrelevant in our arguments and can be easily reinstated in our results, we choose~$\alpha=1$ in \eqref{eq:7} and keep this value throughout the section.

%%%%
\subsubsection{A downward defect inserted in a flat environment}
\label{ssec:1D-zero}

To start with, let us temporarily further simplify the above model. We take~$\ellper\,=\,0$ and recall the trivial example already introduced in~\cite[Section~6]{CPDE-2015}. The Hamilton-Jacobi equation considered then reads as
 \begin{equation}
\label{eq:HJB-1}
u_\varepsilon(x)+\left|(u_\varepsilon)'(x)\right|=\ell_0\left(\frac{x}{\varepsilon}\right),
\end{equation}
and  should be thought as the perturbation, by the local potential~$\ell_0$, of the periodic equation
\begin{equation}
\label{eq:HJB-1-0}
u(x)+\left|u'(x)\right|=0.
\end{equation}
As is well known, the only bounded $C^1$ solution of~\eqref{eq:HJB-1-0} is $u=0$. 
% It of course does not change in the homogenized limit. 
We may likewise make explicit the bounded solution~$u_\varepsilon$ to~\eqref{eq:HJB-1}. For this purpose, and just to keep algebraic manipulations and expressions simple, we assume, in addition to~\eqref{eq:l0-support}-\eqref{eq:l0-downward}, 
that~$\ell_0$ satisfies
\begin{equation}
\label{eq:HJB-3}
\ell_0(0)\,=\,\inf_{x\in\R}\ell_0(x)\quad \text{and}\quad
\left\{
\begin{array}{ll}
\displaystyle \left(\ell_0\right)'(x)< 0,&\forall x\in ]-\frac12,0[\;\cap\; \textrm{Supp}\,\ell_0,\\
&\\
\displaystyle \left(\ell_0\right)'(x)> 0,&\forall x\in ]0,\frac12[\;\cap\; \textrm{Supp}\,\ell_0.
\end{array}
\right.
\end{equation}
Precisely under such conditions, it is easy to verify that the unique $C^1$ solution to~\eqref{eq:HJB-1} reads as
\begin{equation}
\label{eq:HJB-4}
u_\varepsilon(x)=\left\{
\begin{array}{ll}
\displaystyle e^x\left(\ell_0(0)\,+\int_x^0e^{-t}\,\ell_0\left(\frac{t}{\varepsilon}\right)\,dt\right)&\text{for}\;x<0,\\
&\\
\displaystyle e^{-x}\left(\ell_0(0)\,+\int_0^xe^{t}\,\ell_0\left(\frac{t}{\varepsilon}\right)\,dt\right)&\text{for}\;x>0,
\end{array}
\right.
\end{equation}
 Indeed, a simple computation allows to check that, for instance if $x>0$, the function
$$\displaystyle g(x)=e^{x}\,b(x)-b(0)-\int_0^xe^{t}\,b(t)\,dt$$ is such that~$g(0)=0$
and~$g'(x)=e^x\,b'(x)$. Hence, if $b$ reaches its global minimum at $x=0$, is non-increasing for $x<0$ and
non-decreasing for $x>0$, then the function~$g$ satisfies $g(0)=0$, is non-increasing for $x<0$ and
non-decreasing for $x>0$. This clearly implies that $g\geq 0$ on $\mathbb R$. Applying this to~$\displaystyle
b=\ell_0\left(\frac{.}{\varepsilon}\right)$ we deduce that
$$\displaystyle
e^{x}\,\ell_0\left(\frac{x}{\varepsilon}\right)-\ell_0(0)\,-\int_0^xe^{t}\,\ell_0\left(\frac{t}{\varepsilon}\right)\,dt\,\geq
0.$$
This means that, $u_\varepsilon$ being defined by~\eqref{eq:HJB-4}, $(u_\varepsilon)'(x)$ is non-negative for $x>0$,
hence $(u_\varepsilon)'(x) = \left|(u_\varepsilon)'(x)\right|$. 
In addition, since %\eqref{eq:HJB-4} satisfies
$u_\varepsilon(x)+(u_\varepsilon)'(x)=\ell_0\left(\frac{x}{\varepsilon}\right)$,
we obtain equation~\eqref{eq:HJB-1} for $x>0$. A similar argument holds for $x<0$.

Using the explicit expression~\eqref{eq:HJB-4} of the solution to~\eqref{eq:HJB-1}, we now easily identify its limit~$u$ as $\varepsilon\to 0$. Since the function~$\ell_0$ has compact support, the sequence of
functions~$\displaystyle \ell_0\left(\frac{.}{\varepsilon}\right)$ converges strongly to zero in $L^p$ for
any $p<\infty$. Hence, in any such space, the locally uniform limit~$u$ of~$u_\varepsilon$, is
\begin{equation}
\label{eq:HJB-lim}
u(x)=
%\left\{
%\begin{array}{ll}
%\displaystyle e^x\,\ell_0(0)&\text{for}\,x<0,\\
%&\\
%\displaystyle e^{-x}\,\ell_0(0)&\text{for}\,x>0.
%\end{array}
%\right\}\,
\,\ell_0(0)\,e^{-|x|}\,.
\end{equation}
Put differently, the homogenized limit~$u$ is  a bounded (in fact, vanishing at infinity)  solution  to
 \begin{equation}
\label{eq:HJB-5}
\left\{
\begin{array}{l}
\displaystyle u(x)+\left |u'(x)\right |=0\quad\text{for}\,x\not=0,\\
\\
\displaystyle u(0)=\ell_0(0). 
\end{array}
\right.
\end{equation}
It is enlightening to reconcile this particular result with the general result stated in Theorem~\ref{sec:main-result-1}. Since the periodic Hamiltonian identically vanishes here, its homogenized limit is evidently~$\overline H\,=\,0$. The first line of~\eqref{eq:HJB-5} thus agrees with~\eqref{eq:11} (for~the choice of constant~$\alpha\,=\,1$ of course). Far more interesting is the Dirichlet condition~$u(0)=\ell_0(0)$ posed at the origin. In the particular setting of the present section, \eqref{eq:12} and~\eqref{eq:14} respectively read as~$u(0)\,\leq\,-E$  and~$u(0)\,+\,\max\left(E, 0\right)\geq 0$. The effective Dirichlet datum ~$E$ is itself obtained, in whole generality, as the limit, first as~$\lambda\to 0$ and next as~$R\to +\infty$, of~$-\,\lambda w^{\lambda,R}$, where~$w^{\lambda,R}$ is the solution to~\eqref{eq:19}-\eqref{eq:20}. In our setting again, $w^{\lambda,R}$ is actually indeed independent of~$R$ and is easily obtained from the explicit expression~\eqref{eq:HJB-4}. 
Using  a rescaling argument  (setting~$w^{\lambda,R}(x)\,=\,u_{\varepsilon=1}(\lambda x)$ for the rescaled defect~$\displaystyle \lambda^{-1}\ell_0(\lambda^{-1} x)$), we realize that 
\begin{equation}
\label{eq:HJB-wlR}
w^{\lambda,R}(x)=\left\{
\begin{array}{ll}
\displaystyle \lambda^{-1}\,e^{\lambda x}\left(\ell_0(0)\,+\int_{\lambda x}^0e^{-t}\,\ell_0\left(\lambda^{-1} t\right)\,dt\right)&\text{for}\,x<0,\\
&\\
\displaystyle \lambda^{-1}\,e^{-\lambda x}\left(\ell_0(0)\,+\int_0^{\lambda x}e^{t}\,\ell_0\left(\lambda^{-1} t\right)\,dt\right)&\text{for}\,x>0,
\end{array}
\right.
\end{equation}
(actually independent of~$R$) is the bounded $C^1$ solution to the equation
\begin{equation}
\label{eq:HJB-wlR-equation}
\lambda \,w^{\lambda,R}\,+\,\left|(w^{\lambda,R})'\right|\,=\,\ell_0\quad\textrm{in}\;\R.
\end{equation}
%Note that $(w^{\lambda,R})'>0$ on $(0,=\infty)$ and  $(w^{\lambda,R})'<0$ on $(-\infty, 0)$, so \eqref{eq:20} holds.

Taking the limit~$\lambda\to 0$ of~$-\,\lambda \,w^{\lambda,R}$  readily yields the effective Dirichlet condition~$E\,=\, -\,\ell_0(0)$. 
But, since we have assumed that~$\ell_0(0)<0$,  $E$ is positive. Thus
$u(0)\,\leq\,-E$  and~$u(0)\,+\,\max\left(E, 0\right)\geq 0$ combine with one another into
\begin{equation}
\label{eq:1D-E-zero}
u(0)\,=\,-E\,=\,\ell_0(0)\,.
\end{equation}
  We have thus recovered in~\eqref{eq:HJB-5} the general limit~\eqref{eq:11}-\eqref{eq:12}-\eqref{eq:13}-\eqref{eq:14} stated in Theorem~\ref{sec:main-result-1}.
  
  In addition, taking the limit, as $\lambda\to 0$ and say for $x>0$, of %$w^{\lambda,R}(x)\,-\,w^{\lambda,R}(0)$, 
  \begin{equation*}
w^{\lambda,R}(x)\,-\,w^{\lambda,R}(0)= \lambda^{-1}\,\left((e^{-\lambda x}-1)\,\ell_0(0)+e^{-\lambda x}\,\int_0^{\lambda x}e^{t}\,\ell_0\left(\lambda^{-1} t\right)\,dt\right),
\end{equation*}
yields
\begin{equation*}
w^{\lambda,R}(x)\,-\,w^{\lambda,R}(0)\,\to -\,\ell_0(0)\,x\,+\,\int_0^x\ell_0(t)\,dt,
\end{equation*}
and we thus find that, as predicted by~\eqref{eq:31} in Proposition~\ref{sec:ergod-const-glob-2}, $  \lim_{x\to +\infty} w(x)=+\infty$. A similar argument using the first line of~\eqref{eq:HJB-wlR} confirms the same limit in $-\infty$.

\medskip

\begin{remark}
\label{rk:upward}
We note in passing that, had we assumed, say, that the local defect~$\ell_0$ is everywhere \emph{nonnegative}, satisfies 
\begin{equation}
\label{eq:HJB-3-bis}
\ell_0(0)\,=\,\sup_{x\in\R}\ell_0(x)\quad \text{and}\quad
\left\{
\begin{array}{ll}
\displaystyle \left(\ell_0\right)'(x)> 0,&\forall x\in ]-1,0[\;\cap\; \textrm{Supp}\,\ell_0,\\
&\\
\displaystyle \left(\ell_0\right)'(x)<0,&\forall x\in ]0,1[\;\cap\; \textrm{Supp}\,\ell_0,
\end{array}
\right.
\end{equation}
instead of~\eqref{eq:HJB-3} 
 (and assumed also, again a technicality, that~$\ell_0$ is even), then an easy adaptation of the above  algebraic expressions shows that, \textit{mutatis mutandis}, 
\begin{equation}
\label{eq:HJB-4-bis}
u_\varepsilon(x)=\left\{
\begin{array}{ll}
\displaystyle e^{-x}\,\int_{-\infty}^xe^{t}\,\ell_0\left(\frac{t}{\varepsilon}\right)\,dt&\text{for}\;x<0,\\
&\\
\displaystyle e^{x}\,\int_x^{+\infty}e^{-t}\,\ell_0\left(\frac{t}{\varepsilon}\right)\,dt&\text{for}\;x>0.
\end{array}
\right.
\end{equation}
This explicit calculation thus shows that, instead of solving~\eqref{eq:HJB-5}, the homogenized limit~$u$ vanishes over the real line, just like it is the case for the periodic, null, Hamiltonian. As briefly mentioned in the introduction, the presence of this "upward" defect~$\ell_0\geq 0$ only affects the next order term. This relates to the "other regime" considered in the works~\cite{PLL,C-LB-S}.
\end{remark}

\begin{remark}
\label{rk:1D-compact}
We also store for future use (see Remark~\ref{rk:1D-compact-bis} and Section~\ref{sssec:defaut-infini}) the observation that the homogenized limit~$u$ given by~\eqref{eq:HJB-lim} does not have compact support, and therefore only agrees \emph{at infinity}~$x\,=\,\pm\infty$ with the (here, trivial) homogenized limit of the periodic case, solution to~\eqref{eq:HJB-1-0}.
\end{remark}

\medskip

%%%%
\subsubsection{Periodic unperturbed environment: the homogenized limit}
\label{ssec:1D-1-periodic}

We now reinstate a non trivial (that is, non constant) periodic potential~$\ellper$ in~\eqref{eq:addition}, which, without loss of generality, we assume of period~$1$. We also assume that there exists only one point per unit interval where~$\ellper$ reaches its minimum~: $\ellper(x_0)\,=\,\inf_{\R}\ellper$. This is again just for simplicity.  We recall that we additionally assume~\eqref{eq:l0-support}-\eqref{eq:l0-downward} and the separated form~\eqref{eq:separated} of the Hamiltonian~$H$. Before we insert any defect~$\ell_0$ in the periodic environment described by~$\ellper$, we need to lay some groundwork and make explicit the homogenized limit in the absence of defect.

\medskip

\paragraph{Homogenized equation} For convenience, we now recall some basic facts.% The reader already familiar with the periodic theory may of course skip this part and directly proceed to the next paragraph.

We consider the one-dimensional equation
\begin{equation}
\label{eq:HJB-1-lper}
\upereps(x)+\left|(\upereps)'(x)\right|=\ellper\left(\frac{x}{\varepsilon}\right),
\end{equation}
(obviously a particular case of~\eqref{eq:8}) where the subscript~${\rm per}$ in   $\upereps$ refers to the fact that the right-hand side is only the \emph{periodic} cost $\ellper$. Actually, using elementary facts from the theory of viscosity solutions, $\upereps$ is indeed periodic (of period~$\varepsilon$); we will  use this property below, but at this time this is irrelevant.

The homogenized Hamiltonian~$\overline H$ arising from~\eqref{eq:HJB-1-lper} is characterized, for any~$p\in \R$, as the unique real number~$\overline H(p)$ such that there exists a periodic corrector $\chiperp$, viscosity solution to 
\begin{equation}
  \label{eq:9-1D}
 \left | p+ (\chiperp)'(y)\right |\,=\,\overline H(p)\,+\,\ellper(y) \quad \hbox{in }\R\,,
\end{equation}
which is a particular case of~\eqref{eq:9}. Mimicking the explicit calculations performed in~\cite{LPV} that address the case of the quadratic Hamiltonian~$ \left | p\right |^2\,-\,\ellper(y)$ instead of~$ \left | p\right |\,-\,\ellper(y)$, it is easy to identify the periodic solution~$\chiperp$ and the homogenized Hamiltonian%~$\overline H$.
\begin{equation}
 \label{eq:1D-Ham-per}
\overline H(p)\,=\,
\left\{
\begin{array}{ll}
\displaystyle \,-\,\inf_{\R}\ellper&\text{for}\quad|p|\leq \langle\ellper\rangle\,-\,\inf_{\R}\ellper\,,\\
&\\
\displaystyle \,|p|\,-\,\langle\ellper\rangle&\text{for}\quad |p|\geq \langle\ellper\rangle\,-\,\inf_{\R}\ellper.
\end{array}
\right.
\end{equation}
It follows that the (periodic) homogenized equation, defined by~\eqref{eq:8} in whole generality, here reads as
\begin{equation*}
u\,+\,\overline H(u')\,=\,0\,,
\end{equation*}
for the particular Hamiltonian~\eqref{eq:1D-Ham-per}. We immediately notice that its unique BUC viscosity solution is the constant function
\begin{equation*}
u\,=\,\inf_{\R}\ellper\,,
\end{equation*}
which we henceforth denote by~$\uper$ to distinguish it from the homogenized limit~$u$ in the presence of defect, which we are going to manipulate shortly.

\bigskip

\paragraph{  An alternative "direct" proof} Our next step is to remark that  in the present setting,
besides being a consequence of the general periodic homogenization theory, the limit~$\displaystyle \upereps\to\,\inf_{\R}\ellper$  may alternatively and independently be obtained from the original equation~\eqref{eq:HJB-1-lper},
as soon as we have some elementary information on the solution~$\upereps$. 
The interest of following this alternative route is to better understand the phenomena at play and prepare the ground for the homogenization limit in the presence of a defect that will be addressed in Section~\ref{ssec:1D-1-defect-periodic} below.

\medskip

It is immediate to deduce from~\eqref{eq:HJB-1-lper} that its solution~$\upereps$ consists of the combination of two different branches, respectively solving
\begin{equation}
\label{eq:HJB-1-lper+}
\upereps(x)+(\upereps)'(x)\,=\ellper\left(\frac{x}{\varepsilon}\right)\;\quad\textrm{when}\;(\upereps)'(x)\geq 0
\end{equation}
and
\begin{equation}
\label{eq:HJB-1-lper-}
\upereps(x)-(\upereps)'(x)\,=\ellper\left(\frac{x}{\varepsilon}\right)\;\quad \textrm{when}\;(\upereps)'(x)\leq 0\,.
\end{equation}
We now introduce
\begin{equation}
\label{eq:gper}
\gpereps(x)\,=\,e^{-x}\,\int_{-\infty}^xe^t\,\ellper\left(\frac{t}{\varepsilon}\right)\,dt,
\end{equation}
which is not an unexpected function given the calculations in Section~\ref{ssec:1D-zero} above. It is evidently a periodic function, which satisfies~$\displaystyle \gpereps\,+\,(\gpereps)'\,=\,\ellper\left(\frac{.}{\varepsilon}\right)$ and thus $\langle \gpereps\rangle\,=\,\langle \ellper\rangle$. 
We deduce  that in the branch described by ~\eqref{eq:HJB-1-lper+},  $\upereps(x)\,=\,\gpereps(x)\,+\,\alpha\,e^{-x}$ for some constant~$\alpha$ (depending on~$\varepsilon$ but this is irrelevant). 

Should the branch~\eqref{eq:HJB-1-lper+} extend to~$+\infty$, we would therefore obtain, using periodicity and shifting to $+\infty$, that  $\upereps(x)\,=\,\gpereps(x)$ for all~$x$ in that branch. 
Since~$\langle \gpereps\rangle\,=\,\langle \ellper\rangle$ while, in view of~\eqref{eq:HJB-1-lper},  $\displaystyle \upereps\leq\ellper\left(\frac{.}{\varepsilon}\right)$ everywhere, this implies~$\displaystyle \gpereps\,=\, \ellper\left(\frac{.}{\varepsilon}\right)$. 
Given the differential equation satisfied by~$\gpereps$, this may only hold when~$\ellper$ is constant. 
We reach a contradiction with the assumption made on $\ellper$.

Similarly to~\eqref{eq:gper}, we introduce the function~$\displaystyle \hpereps(x)\,=\,e^{x}\,\int_x^{\infty}e^{-t}\ellper\left(\frac{t}{\varepsilon}\right)\,dt$, and get that in the branch~\eqref{eq:HJB-1-lper-}, $\upereps(x)\,=\,\hpereps(x)\,+\,\beta\,e^{x}$. 
It immediately follows that~\eqref{eq:HJB-1-lper-} cannot extend to~$+\infty$ either, since the asymptotic blow-up implied by the term~$e^{x}$  would contradict the boundedness of the solution.

A symmetric argument  allows to  exclude that the branches extend to~$-\infty$. We conclude that the solution~$\upereps$ consists of an effective, infinite alternation of the two branches~\eqref{eq:HJB-1-lper+} and~\eqref{eq:HJB-1-lper-}.  
In our particular periodic setting, we \emph{could} have obtained this fact simply using the uniqueness, thus the periodicity, of the viscosity solution to~\eqref{eq:HJB-1-lper}. 
But the specific argument we have developed will be useful later, in the presence of a defect. 

We next remark that, because of the generic properties of viscosity solutions, a transition, at some~$\overline x$, from the branch~\eqref{eq:HJB-1-lper-} at the left of~$\overline x$ to the branch~\eqref{eq:HJB-1-lper+} at its right, may only occur when ~$(\upereps)'$ is continuous at~$\overline x$, nonpositive before and nonnegative after~$\overline x$. 
Thus~$(\upereps)'(\overline x)\,=\,0$ and~$\overline x$ is a local minimizer of~$\upereps$. 
Using the equation, $\upereps(\overline x)\,=\,\ellper(\overline x)$. 
Since we know that~$\upereps\leq \ellper$ everywhere, the three facts altogether imply that~$\ellper$ reaches its minimum at~$\overline x$. 
For simplicity, we have assumed that such a minimizer is unique per period. So we exactly know that, per period, there is one transition~\eqref{eq:HJB-1-lper-}-\eqref{eq:HJB-1-lper+} (at the minimizer~$\overline x=\argmin\ellper$) and, thus also, one transition~\eqref{eq:HJB-1-lper+}-\eqref{eq:HJB-1-lper-} (thus necessarily \emph{above} the level~$\min_\R\ellper$).

An interesting corollary of the previous observation is the following. 
Since $\upereps\geq\inf_\R\ellper$ everywhere,  $\left|(\upereps)'\right|\,=\,\ellper\,-\,\upereps$ is thus bounded independently of~$\varepsilon$. 
It follows from these two facts together with the periodicity of  $\upereps$  (with period~$\varepsilon$) that, necessarily, $\upereps\to \inf_\R\ellper$ as~$\varepsilon$ vanishes. In dimension~1, this provides us with an independent and alternative proof of the homogenization limit, as announced above.

%%%%
\subsubsection{A downward defect in this periodic environment}
\label{ssec:1D-1-defect-periodic}

We now insert the defect~$\ell_0$. The equation under consideration is 
\begin{equation}
\label{eq:HJB-1-lper-def}
u_\varepsilon(x)+\left|(u_\varepsilon)'(x)\right|=\ellper\left(\frac{x}{\varepsilon}\right)\,+\,\ell_0\left(\frac{x}{\varepsilon}\right).
\end{equation}
Using the exact same argument as above in~\eqref{eq:HJB-1-lper+}-\eqref{eq:HJB-1-lper-}, the two auxiliary functions~$ \gpereps$ and~$ \hpereps$, along with the  functions $\displaystyle\gzeroeps(x)\,=\,e^{-x}\,\int_{-\infty}^xe^{t}\ell_0\left(\frac{t}{\varepsilon}\right)\,dt$, ~$\displaystyle\hzeroeps(x)\,=\,e^{x}\,\int_x^{\infty}e^{-t}\ell_0\left(\frac{t}{\varepsilon}\right)\,dt$, we may similarly conclude that no branch can  extend to~$\pm\infty$. 
Recall indeed that $\ell_0$ vanishes at infinity so  does not modify the argument performed in the periodic case above.

We can thus claim that $u_\varepsilon$ is an infinite alternation of the two branches. Again similarly as above, we know that the junction~\eqref{eq:HJB-1-lper-}-\eqref{eq:HJB-1-lper+}  (in this order from left to right) may only occur at local minimizers of~$\displaystyle (\ellper+\ell_0)\left(\frac{.}{\varepsilon}\right)$, and that at such points, $\displaystyle u_\varepsilon\,=\,(\ellper+\ell_0)\left(\frac{.}{\varepsilon}\right)$. 
Outside~$\displaystyle \textrm{Supp}\,\ell_0\left(\frac{.}{\varepsilon}\right)$, this thus occurs at some minimizers of~$\displaystyle \ellper\left(\frac{.}{\varepsilon}\right)$, at which~$\displaystyle u_\varepsilon$ coincides with~$\ellper\left(\frac{.}{\varepsilon}\right)$. 

We next recall that, precisely at all minimizers of~$\displaystyle \ellper\left(\frac{.}{\varepsilon}\right)$, we also  have the equality~$\displaystyle \upereps\,=\,\ellper\left(\frac{.}{\varepsilon}\right)$, where $\upereps$ denotes the solution in the absence of defects. 
Thus, at any minimizer of~$\displaystyle \ellper\left(\frac{.}{\varepsilon}\right)$ outside~$\displaystyle \textrm{Supp}\,\ell_0\left(\frac{.}{\varepsilon}\right)$ where~$\displaystyle u_\varepsilon$ coincides with~$\ellper\left(\frac{.}{\varepsilon}\right)$, we \emph{also} have~$u_\varepsilon\,=\,\upereps$. Between two such minimizers, using a comparison principle for the equation, we deduce that~$u_\varepsilon\,\equiv\,\upereps$. So $u_\varepsilon\,\equiv\,\upereps$ at least everywhere outside a bounded interval containing the origin and presumably slightly exceeding~$\textrm{Supp}\,\ell_0\left(\frac{.}{\varepsilon}\right)$. 
%On that interval, the homogenized limit of~$u_\varepsilon$ would agree with that of~$\upereps$, namely~$\inf_\R\ellper$.
%The point is however that this interval may vary with~$\varepsilon$, and that we also need to investigate the behavior of~$u_\varepsilon$ \emph{inside}.

Our next task is thus to fully identify the homogenized limit of~$u_\varepsilon$. Since we do not have as much information on~$u_\varepsilon$ as we used to have on~$\upereps$ in the previous section, we need to make a small detour (in fact by defining the homogenized problem itself and next backpedalling) .

\medskip

We introduce the function
\begin{equation}
  \label{eq:1D-limit-homog-sol}
u(x)\,=\,
\left\{
\begin{array}{ll}
\displaystyle \langle\ellper\rangle\,-\,e^{-|x|}\,\left(\langle\ellper\rangle\,-\,\inf_{\R}(\ellper\,+\,\ell_0)\right) &\text{for}\,|x|\leq \mu,\\
&\\
\displaystyle \inf_\R\ellper&\text{for}\,|x|\geq \mu\,,
\end{array}
\right.
\end{equation}
for~$\mu$ defined by
\begin{equation}
\label{eq:1D-mu}
e^{-\mu}\,=\,\frac{\langle\ellper\rangle\,-\,\inf_{\R}\ellper}{\langle\ellper\rangle\,-\,\inf_{\R}(\ellper\,+\,\ell_0)}.
\end{equation}
This function~$u$ is indeed the homogenized limit of~$u_\varepsilon$, as we will see below. 
For the time being, we just remark that
\begin{description}
\item[(i)] when~$\ell_0\equiv 0$ and more generally exactly when condition~\eqref{eq:l0-downward} is not satisfied, then~$\mu=0$ 
%(we decide by convention  that $\mu=0$ also when the right-hand side of~\eqref{eq:1D-mu} is outside~$]0,1]$) 
and $u\,=\,\inf_\R\ellper$ is indeed the homogenized limit we found for~$\upereps$. 
\item[(ii)] when~$\ellper\equiv 0$, then~$\mu\,=\,+\,\infty$, $u(x)\,=\,(\inf_{\R}\ell_0)\,e^{-|x|}$ % \,=\,\ell_0(0)\,e^{-|x|}$ 
and we recover the homogenized limit~\eqref{eq:HJB-lim} if we additionally assume~\eqref{eq:HJB-3} as we did in Section~\ref{ssec:1D-zero}.
\end{description}

\medskip

\noindent We also remark that~$u$ solves
\begin{equation*}
%  \label{eq:1D-limit-homog-eq}
\left\{
\begin{array}{l}
\displaystyle u(x)\,+\,\overline{H}\left(u'(x)\right)=0\quad\text{for}\,x\not=0,\\
\\
\displaystyle u(0)\,=\,\inf_{\R}(\ellper\,+\,\ell_0), %,\\
\end{array}
\right.
\end{equation*}
where~$\overline{H}$ is defined in~\eqref{eq:1D-Ham-per}.  In this specific setting, we again recover, as in Section~\ref{ssec:1D-zero}, the general limit~\eqref{eq:11}-\eqref{eq:12}-\eqref{eq:13}-\eqref{eq:14} stated in Theorem~\ref{sec:main-result-1}, this time with the specific value
\begin{equation}
\label{eq:1D-E}
E\,=\,-\,\inf_{\R}(\ellper\,+\,\ell_0)
\end{equation}
of the effective Dirichlet condition. Expression~\eqref{eq:1D-E} is obviously a generalization of~\eqref{eq:1D-E-zero}. And \eqref{eq:12} and~\eqref{eq:14} again combine with one another to yield the above Dirichlet condition, precisely because of the condition~\eqref{eq:l0-downward}. We omit here the detailed verification that the explicit value of~$E$ given in~\eqref{eq:1D-E} may indeed be independently obtained using the sequence of approximate correctors, as in the general theory using~\eqref{eq:19}-\eqref{eq:20}, similarly to what we did above in Section~\ref{ssec:1D-zero} in the particular setting where~$\ellper$ identically vanishes. 
% We have thus recovered in~\eqref{eq:HJB-5} 

\medskip

 In order to prove that~$u$ defined in~\eqref{eq:1D-limit-homog-sol} is indeed the limit of~$u_\varepsilon$, let us consider
  \begin{equation*}
v_\varepsilon(x)\,=\,e^{-x}\int_0^xe^t\,\ell\left(\frac{t}{\varepsilon}\right)\,dt\,+e^{-x}\, \inf_{\R}\ell.
\end{equation*}
As~$\varepsilon$ vanishes, this function pointwise (and in fact  locally uniformly) converges to
 \begin{equation*}
\langle\ellper\rangle \,(1-e^{-x})\,+e^{-x}\, \inf_{\R}\ell\, =\,u(x),
\end{equation*}
as defined by~\eqref{eq:1D-limit-homog-sol}, for all $0<x<\mu$. 
Since evidently~$\displaystyle\ellper\left(\frac{x}{\varepsilon}\right)\,\geq\,\inf_{\R}\ellper$, 
\begin{equation*}
\ellper\left(\frac{x}{\varepsilon}\right)\,-\,v_\varepsilon(x)\,\geq\,\inf_{\R}\ellper\,-\,e^{-x}\int_0^xe^t\,\ell\left(\frac{t}{\varepsilon}\right)\,dt\,-e^{-x}\, \inf_{\R}\ell,
\end{equation*}
where, as~$\varepsilon$ vanishes, the right-hand side likewise converges to
 \begin{equation*}
\inf_{\R}\ellper\,-\,\,\langle\ellper\rangle \,(1-e^{-x})\,-e^{-x}\,\inf_{\R}\ell \,=\,(\inf_{\R}\ellper\,-\,\inf_{\R}\ell)\,-\,(\langle\ellper\rangle\,-\,\inf_{\R}\ell)\,(1-e^{-x}).
\end{equation*}
Note that   this quantity is nonnegative when~$x\leq\mu$, since
\begin{eqnarray*}
&&(\inf_{\R}\ellper\,-\,\inf_{\R}\ell)\,-\,(\langle\ellper\rangle\,-\,\inf_{\R}\ell)\,(1-e^{-x})\cr
&&\hskip 2truecm\geq(\inf_{\R}\ellper\,-\,\inf_{\R}\ell)\,-\,(\langle\ellper\rangle\,-\,\inf_{\R}\ell)\,(1-e^{-\mu})\cr
&&\hskip 2truecm=0,
\end{eqnarray*}
and only vanishes at~$x=\mu$.

If we now take~$0\,<\, x\,<\,\mu$, outside~$\displaystyle\textrm{Supp}\,\ell_0\left(\frac{.}{\varepsilon}\right)$, and slightly bounded away from~$\mu$ (by some irrelevant constant, say of order~$\varepsilon$), we thus have, for~$\varepsilon$ sufficiently small,
\begin{equation*}
\ell\left(\frac{x}{\varepsilon}\right)\,-\,v_\varepsilon(x)\,\geq (\textrm{actually}\,=\,)\,\ellper\left(\frac{x}{\varepsilon}\right)\,-\,v_\varepsilon(x)\,>\,0.
\end{equation*}
If follows that, in that region, $\displaystyle(v_\varepsilon)'(x)\,=\,\ell\left(\frac{x}{\varepsilon}\right)\,-\,v_\varepsilon\,>\,0$ and thus $v_\varepsilon$ is solution to~\eqref{eq:HJB-1-lper-def}. We may proceed similarly for~$-\mu\,<\, x\,<\,0$ and the other branch, that of type~\eqref{eq:HJB-1-lper-}. The function~$v_\varepsilon$ that we have constructed may then be continued close to and beyond~$\pm\mu$. Eventually,  it yields the global solution~$u_\varepsilon$.

\begin{remark}
\label{rk:1D-compact-bis}
In echo to  Remark~\ref{rk:1D-compact}, we notice that, here, the two homogenized limits obtained respectively for $\ellper$ and~$\ellper+\ell_0$ \emph{exactly} agree outside a bounded interval surrounding the origin, unless~$\ellper$ is constant. The latter case is indeed the only case where the homogenized Hamiltonian~$\overline H$ is not flat %degenerate
around~$p=0$ (it is~$|p|$). The agreement at infinity of the two homogenized solutions, respectively in the absence and in the presence of defects, reflects this property. We will further investigate this question in higher dimensions in Section~\ref{sssec:defaut-infini} below.
\end{remark}

% section probas
\subsubsection{A randomized variant}
\label{ssec:proba}

As briefly mentioned in Section~\ref{sec:introduction}, a randomized variant of the theory of local defects has been introduced by A.~Anantharaman and the second author in~\cite{Anantha-LB}. In that work, the equation under consideration was a linear elliptic equation in conservation  form. The adaptation of the setting considered therein to our equation~\eqref{eq:HJB-1-lper-def} is as follows. 

We consider as a perturbation of the underlying periodic environment encoded in~$\ellper$, the \emph{random} running cost
\begin{equation*}
\ell_S(x,\omega)\,=\,\sum_{k\in\Z}X^\eta_k(\omega)\,\ell_0(x-k),
\end{equation*}
where $\displaystyle\{X^\eta_k(\omega)\}_{k\in\Z}$ is a sequence of i.i.d. random variables that all follow a Bernoulli law of parameter~$\eta$, that is, $X^\eta_k(\omega)\,=\,0$ or $\,=\,1$ with probability $1-\eta$ and~$\eta$ respectively.  The parameter~$\eta>0$ is to be thought of as a small parameter, which will eventually be sent to zero. This models that the periodic environment encoded in~$\ellper$ is only \emph{slightly} perturbed.

After rescaling in~$\varepsilon$ of the potential~$\ell_S$, the equation for which we study homogenization thus reads as
\begin{equation*}
u_\varepsilon(x,\omega)+\left|(u_\varepsilon)'(x,\omega)\right|=\ellper\left(\frac{x}{\varepsilon}\right)\,+\,\ell_S\left(\frac{x}{\varepsilon}\,,\omega \right)\,=\ellper\left(\frac{x}{\varepsilon}\right)\,+\,\,\sum_{k\in\Z}X^\eta_k(\omega)\,\ell_0\left(\frac{x-k\varepsilon}{\varepsilon}\right).
\end{equation*}
We intend to identify the homogenized limit~$\varepsilon\to 0$ for this equation \emph{in the regime when the  parameter~$\eta$  vanishes}. Here, to keep things simple and allow for analytic calculations, we only consider  the case~$\ellper\,\equiv \,0$  of a flat unperturbed environment, that is
\begin{equation}
\label{eq:random1-flat}
u_\varepsilon(x,\omega)+\left|(u_\varepsilon)'(x,\omega)\right|=\,\sum_{k\in\Z}X^\eta_k(\omega)\,\ell_0\left(\frac{x-k\varepsilon}{\varepsilon}\right),
\end{equation}
using a "shape function" $\ell_0$ for the defects that satisfies~\eqref{eq:HJB-3}. 
\begin{remark}
Let us  emphasize that, even though the specific algebraic manipulations below depend on the technical assumption that $\ellper=0$, the general result that we are going to obtain carries over to other cases.
% In particular, if $d=1$ and  $\ellper$ is not identically zero, we know that 
%  there exist a positive integer $r$ and a real number $L>0$ 
%    such that on the one hand, $u_\varepsilon$ and $\upereps$ agreee for $|x|> r\varepsilon$, and on the other hand,  $\upereps$ tends to $\inf_{\R} \ellper> -E$ uniformly. These two facts make it possible to generalize the results written  below to  perturbed Hamiltonians  of the form $|p|- \ellper\left(\frac{x}{\varepsilon}\right)\,-\,\,\sum_{k\in\Z}X^\eta_k(\omega)\,\ell_0\left(\frac{x-kq \varepsilon}{\varepsilon}\right)$, for a sufficiently large integer  $q$.

Extensions to  other ambient dimensions will be addressed in Section~\ref{ssec:others}.
\end{remark}

In order to understand the homogenized limit, we need a preliminary step, namely the case of \emph{two} localized (downward) defects, instead of a single one at the origin as in Section~\ref{ssec:1D-1-defect-periodic}. Intuitively, $\ell_S(x,\omega)$ is the random superposition of an infinite number of such defects. Since \emph{"He who can do more can do less"}, we need to first understand the case of two defects, that is:
\begin{equation}
\label{eq:equa-2}
z_\varepsilon(x)+\left|(z_\varepsilon)'(x)\right|=\,\ell_0\left(\frac{x}{\varepsilon}\right)\,+\,\ell_0\left(\frac{x-\varepsilon}{\varepsilon}\right).
\end{equation}
It is then easy to realize that the solution reads as
\begin{equation}
\label{eq:min-2}
z_\varepsilon(x)\,=\,\inf\left(u_\varepsilon(x)\,,\,u_\varepsilon(x-\varepsilon)\right),
\end{equation}
where~$u_\varepsilon$ is the solution for one single defect at the origin, which we made explicit in~\eqref{eq:HJB-4}.  
In order to realize that~\eqref{eq:min-2} holds, we just need to establish that, in~$\displaystyle\textrm{Supp}\,\ell_0\left(\frac{\cdot}{\varepsilon}\right)$, we have $u_\varepsilon(x)\,\leq\,u_\varepsilon(x-\varepsilon)$. A similar argument will prove a symmetric result on~$\displaystyle\textrm{Supp}\,\ell_0\left(\frac{x-\varepsilon}{\varepsilon}\right)$, while, outside the two supports, both~$u_\varepsilon(\cdot)$ and~$u_\varepsilon(\cdot-\varepsilon)$ solve~\eqref{eq:equa-2}. 
Let us recall that, because of~\eqref{eq:l0-support}, the two supports~$\displaystyle\textrm{Supp}\,\ell_0$ and $\displaystyle\textrm{Supp}\,\ell_0(\cdot -1)$ are disjoint. 
Given the explicit expression~\eqref{eq:HJB-4} of $u_\varepsilon$ and its monotonicity on~$\R^+$,  we then have that 
\begin{equation}
\label{eq:min-for1}
u_\varepsilon(x)\leq e^{-\varepsilon\,a}\left(\ell_0(0)\,+\int_0^{\varepsilon\,a}e^{t}\,\ell_0\left(\frac{t}{\varepsilon}\right)\,dt\right),
\end{equation}
for~$a\leq\frac12$ and all~$0<x<\varepsilon\,a$. On the other hand, again for all~$0<x<\varepsilon\,a$, we have, by the same argument but this time looking at the explicit expression of~$u_\varepsilon(x-\varepsilon)$ that , 
\begin{equation}
\label{eq:min-for2}
u_\varepsilon(x-\varepsilon)\geq e^{\varepsilon\,a-\varepsilon}\left(\ell_0(0)\,+\int_{\varepsilon-\varepsilon\,a}^\varepsilon e^{-(t-\varepsilon)}\,\ell_0\left(\frac{t-\varepsilon}{\varepsilon}\right)\,dt\right).
\end{equation}
Since~$a\leq\frac12$, we have $e^{\varepsilon\,a-\varepsilon}\,\leq\, e^{-\varepsilon\,a}$. On the other hand, for simplicity let us further assume that
\begin{equation}
\label{eq:even}
\ell_0\quad\textrm{is an even function}\,, 
\end{equation}
 so that the two integrals  appearing in the right-hand sides of~\eqref{eq:min-for1}
  and~\eqref{eq:min-for2}  are identical. We deduce that 
$$u_\varepsilon(x)\leq  u_\varepsilon(x-\varepsilon)\quad\textrm{for all}\;0<x<\varepsilon\,a,$$
that region being the part of~$\displaystyle\textrm{Supp}\,\ell_0\left(\frac{\cdot}{\varepsilon}\right)$ at the right of the origin. A similar argument allows to conclude for all the other regions. Note that a simple estimation of the terms in~\eqref{eq:min-for1} and~\eqref{eq:min-for2} shows that, had we not assumed~\eqref{eq:even} for simplicity, the argument above would also hold true provided the two defects are separated from a distance $q\,\varepsilon$ for an integer $q$ chosen sufficiently large. We will see a similar argument
% in~\eqref{eq:conditions-spherique}
in Section~\ref{sssec:proba-dD} below. 
% les deux integrales sont  $O(\varepsilon)$ avec des constantes qui en d{\~A}{\copyright}pendent que de \ell_0, et  l'in{\~A}{\copyright}galit{\~A}{\copyright} devient $e^{\varepsilon\,a-\varepsilon\,k}\,(1+ A \varepsilon) \leq\, e^{-\varepsilon\,a}\,(1+ B \varepsilon)$, donc on prend k grand devant A-B.

\medskip

Equation~\eqref{eq:min-2}  readily implies that, returning to the random setting~\eqref{eq:random1-flat}, 
\begin{equation}
\label{eq:random-min-formula}
u_\varepsilon(x, \omega)\,=\,\inf_{k\in\Z}\left\{X^\eta_k(\omega)\,u_\varepsilon(x-k\varepsilon)\right\},
\end{equation}
since any $k\in\Z$ for which $X^\eta_k(\omega)=1$ contributes to lowering the solution, while others~$k$ do not. The "min-formula"~\eqref{eq:random-min-formula} is the main ingredient in the rest of our argument and will be generalized in Section~\ref{ssec:others}.

\medskip

\begin{remark} Note that we work here with the solution $u_\varepsilon$ of the original equation~\eqref{eq:random1-flat}, prove it explicitly reads as~\eqref{eq:random-min-formula} and proceed from there to find its homogenized limit. In fact, we could, using a similar string of arguments, identify the solution to the discounted problem
\begin{equation}
\label{eq:random1-flat-corrector}
\delta\,w_\delta(y,\omega)+\left|(w_\delta)'(y,\omega)\right|=\,\sum_{k\in\Z}X^\eta_k(\omega)\,\ell_0\left(y-k\right),
\end{equation}
for $\delta >0$, that replaces the corrector problem in the context of stochastic homogenization. The similarity between~\eqref{eq:random1-flat} and~\eqref{eq:random1-flat-corrector} is evident. Using a simple rescaling (exactly as we did in the specific deterministic case~\eqref{eq:HJB-wlR}-\eqref{eq:HJB-wlR-equation} above), any explicit expression of the solution~$u_\varepsilon$ of the former equation yields an expression of the solution~$w_\delta$ to the latter equation. Taking the limit~$\lim_{\delta\to 0} \delta\,w_\delta$, we may then identify the value ${\overline H}_\eta(0)$ of the homogenized Hamiltonian in this stochastic setting. 
Note that the same remark applies to the multidimensional context we will address in Section~\ref{sssec:proba-dD} below, but we will not repeat it there.
\end{remark}

Since we already know that, in the homogenized limit, $u_\varepsilon$ given by~\eqref{eq:HJB-4} converges to~$u(x)=\,\ell_0(0)\,e^{-|x|}$ in~\eqref{eq:HJB-lim} 
(later generalized in~\eqref{eq:1D-limit-homog-sol}), we \emph{temporarily} replace~\eqref{eq:random-min-formula} by
\begin{equation}
\label{eq:random3-flat}
u_\varepsilon(x)\,=\,\ell_0(0)\,\sup_{k\in\Z}\left\{X^\eta_k(\omega)\,e^{-|x-k\varepsilon|}\right\},
\end{equation}
where we have used that~$\ell_0(0)<0$. Take now~$x=0$. The key quantity appearing on the right-hand side of~\eqref{eq:random3-flat} is in fact the random variable
\begin{equation}
\label{eq:random-flat-Z0}
 \sup_{k\in\Z}\left\{X^\eta_k(\omega)\,e^{-|k|\varepsilon}\right\},
 \end{equation} 
and to get a grasp on this quantity, let us only consider therein the positive indices, so that we focus our attention on
\begin{equation*}
%\label{eq:random-flat-Z1}
Z^\eta_\varepsilon(\omega)\,=\,\sup_{k\in\Z,k\geq 0}\left\{X^\eta_k(\omega)\,e^{-k\varepsilon}\right\}.
\end{equation*}
Since the function~$k\to e^{-k\varepsilon}$ is decreasing with respect to $k>0$,   it is evident that~$Z^\eta_\varepsilon\,=\,e^{-k\varepsilon}$ for the first index~$k$ such that~$X^\eta_k=1$. But now, since the $X^\eta_k$ are i.i.d. and all follow a Bernoulli law, it is well known that the law of this first index is geometric. More precisely, we have
\begin{equation}
\label{eq:random-flat-Z2}
Z^\eta_\varepsilon(\omega)\,=\,e^{-k\varepsilon}\quad\textrm{with probability}\;(1-\eta)^k\eta.
\end{equation}
If we now let~$\varepsilon\to 0$, we obtain $\displaystyle Z^\eta_\varepsilon(\omega)\,\to\, 1$ almost surely.
It is easy to see that a similar argument applies to~\eqref{eq:random-flat-Z0}, since only~$|k|$ matters there and~$\max(X^\eta_{-k},X^\eta_k)$ are i.i.d. Bernoulli variables of parameter~$1-(1-\eta)^2$, which scales as~$\eta$, when~$\eta$ is small. Thus, eventually we may guess that
\begin{equation*}
u_\varepsilon(0, \omega)\,\to\,\ell_0(0)
\end{equation*}
almost surely in the  limit~$\varepsilon\to 0$. The actual proof of this fact necessitates to realize that we have replaced $u_\varepsilon$ given by~\eqref{eq:HJB-4} by its homogenized limit~$u(x)=\,\ell_0(0)\,e^{-|x|}$. The error committed in this approximation may easily be controlled (at least in this case where everything is explicitly known by the formulae~\eqref{eq:HJB-4} and~\eqref{eq:HJB-lim}) and proven to be irrelevant in the limit process above.  In addition, the asymptotics that we have found for $x=0$ in~\eqref{eq:random3-flat} readily carry over to any~$x\in \R$, since all what matters is the  decay of the function~$k\to e^{-|x-k\varepsilon|}$ away from $x$. We have thus obtained
\begin{equation}
\label{eq:random4-flat}
\textrm{a.s.}\,\lim_{\varepsilon\to 0}u_\varepsilon(x, \omega)\,=\,\ell_0(0).
\end{equation}
The convergence~\eqref{eq:random4-flat} is both \emph{disappointing} and \emph{intuitive}. 
It is disappointing because, the homogenized limit being flat, there is not much interesting  mathematical phenomenon to discuss. The convergence~\eqref{eq:random4-flat} is, on the other hand,  intuitive, because as soon as a defect occurs at a location~$k\varepsilon$ for some index~$k$ such that $k\varepsilon\to 0$ in the homogenized limit, this defect is brought to the origin by the rescaling and everything happens \emph{as if} we had a deterministic localized defect there. For instance, even the  probability of having \emph{at least one} defect in the~$\displaystyle\frac{1}{\sqrt\varepsilon}$  first indices (say at the right of the origin)  is
\begin{equation}
\label{eq:random41-flat}
\displaystyle \sum_{0\leq k\leq\frac{1}{\sqrt\varepsilon}}(1-\eta)^k\,\eta\,\approx\, 1- (1-\eta)^{1/{\sqrt\varepsilon}}\,\buildrel \varepsilon\to 0\over\longrightarrow \,1,
\end{equation}
for $\eta$ fixed, so all this happens with probability one.
A similar argument again applies if we consider both sides of the origin, and likewise when we consider any point~$x$ since all points play identical roles. In the end, the limit~\eqref{eq:random4-flat} is flat. It is also independent of~$\eta$.

This comes in sharp contrast to the case of a similar random coefficient inserted in an elliptic equation, which is the setting studied in the work~\cite{Anantha-LB}.  This striking difference may be intuitively explained by the fact that, at the macroscopic scale, an elliptic equation only \emph{sees} suitable \emph{averages}  of the oscillatory coefficient, while an Hamilton-Jacobi type equation sensitively \emph{sees extremal values} of that coefficient. Put in an even more simplified language, what we observe is nothing but the difference between
\begin{equation*}
\lim_{N\to\infty}\frac1N\sum_{k=1}^NX^\eta_k(\omega)\quad\textrm{and}\quad\,\lim_{N\to\infty}\sup_{1\leq k\leq N}\left\{X^\eta_k(\omega)\right\},
\end{equation*}
for the i.i.d. Bernoulli random variables~$X^\eta_k$ of parameter~$\eta$. The leftmost limit scales as~$\eta$ while the rightmost limit is independent of~$\eta$ . 

\medskip

The net conclusion of the above discussion is that, in the model we have considered so far, the probability of having defects is so large that, in the limit, there are almost surely defects everywhere. In order to get a much more interesting regime, we need to relate the small parameter~$\varepsilon$, which measures the distances at the microscopic scale, with the parameter~$\eta$, which measures the \emph{amount of random perturbation present in the environment} (recall that, for $\eta=0$, we recover the unperturbed environment). A possible way to obtain this property is to ensure that, for instance, there is, asymptotically, a fixed, small, proportion of defects per unit \emph{macroscopic} length (or, in higher dimensional settings, volume). More precisely, we now take
\begin{equation}
\label{eq:random5-flat}
\eta\,=\,\overline{\eta}\,\varepsilon,\quad\textrm{for}\quad\overline{\eta}\quad\textrm{fixed}.
\end{equation}
The scaling law~\eqref{eq:random5-flat} is precisely adjusted so that, in contrast to what we observed above in~\eqref{eq:random41-flat} for $\eta$ fixed, the probability of seeing at least one defect over a unit macroscopic length around the origin scales as
\begin{equation}
\label{eq:random51-flat}
\displaystyle \sum_{0\leq k\leq\frac1\varepsilon}(1-{\overline\eta}\varepsilon)^k\,{\overline\eta}\varepsilon\,\approx\,1- (1-{\overline\eta}\varepsilon)^{1/\varepsilon}\,\buildrel \varepsilon\to 0\over\longrightarrow \,1- e^{-{\overline\eta}}\,\buildrel{\overline\eta}\to 0\over\approx{\,\overline\eta}\,\ll\,1.
\end{equation}
If we revisit our calculations above with the particular value~\eqref{eq:random5-flat}, we may mimic the argument step by step. We realize that the key quantity from~\eqref{eq:random-flat-Z2} is now
\begin{equation}
\label{eq:random-flat-Z3}
Z^\eta_\varepsilon(\omega)\,=\,e^{-k\varepsilon}\quad\textrm{with probability}\;(1-{\overline\eta}\varepsilon)^k\,{\overline\eta}\,\varepsilon.
\end{equation}
A simple calculation shows that, for all~$\mu\geq 0$,
\begin{equation*}
{\mathbb P}\left(Z^\eta_\varepsilon\leq e^{-\mu}\right)\,=\,(1-{\overline\eta}\varepsilon)^{\mu/\varepsilon}\,\approx\,e^{-{\overline\eta}\,\mu},\quad\textrm{as}\quad\varepsilon\to 0,
\end{equation*}
that is, for all $0\leq t\leq 1$,
\begin{equation}
{\mathbb P}\left(Z^\eta_\varepsilon\leq t\right)\,\approx\,t^{\overline\eta},\quad\textrm{as}\quad\varepsilon\to 0.
\end{equation}
Upon differentiating with respect to~$t$, we obtain that, in the limit~$\varepsilon\to 0$, the law is~$\displaystyle{\overline \eta}\,{t^{\overline\eta-1}}$. This suffices to establish that the limit is not deterministic. Put differently, (i) when~$\eta>0$ is fixed (or~$\eta\gg\varepsilon$) the limit~$u_\varepsilon$  is deterministic and flat at the level~$\ell_0(0)$, (ii) when ~$\eta=0$ (or~$\eta\ll\varepsilon$) the limit~$u_\varepsilon$  is deterministic and flat at the level~$0$ (that is the unperturbed solution), and (iii) when exactly $\eta$ scales as~$\overline{\eta}\varepsilon$ in~\eqref{eq:random5-flat}, then the limit is flat, and its value is a random variable in the interval $[\ell_0(0),0]$.

\medskip

\subsection{Generalization to higher dimensions}
\label{ssec:others}

\medskip

Motivated by the phenomena we have just observed in the one-dimensional setting of  Section~\ref{sec:1D}, we consider  two specific questions in a higher dimensional situation: 
\begin{description}
\item[(i)] the behavior when $|x|\to +\infty$ of the homogenized solution in presence of a defect at the origin, as compared to that of the homogenized solution in the unperturbed periodic setting
\item[(ii)]  the random superposition of point defects.
\end{description}
The two issues are respectively studied in Sections~\ref{sssec:defaut-infini} and~\ref{sssec:proba-dD} below.

\medskip

\subsubsection{Influence of the defect at infinity}
\label{sssec:defaut-infini}

\medskip

We return to the general setting of Section~\ref{sec:setting} with, in ambient dimension~$d\geq 1$, a general Hamiltonian~$H$ (defined in~\eqref{eq:1}) satisfying the classical properties made precise in that section, in particular~\eqref{eq:2} through~\eqref{eq:4} and that is the perturbation of a periodic Hamiltonian~$\Hper$ (defined in~\eqref{eq:5}) in the following sense: $f=\fper$ and~$\ell=\ellper$ outside a neighborhood of the origin.

We denote by~$u$ the homogenized solution provided by Theorem~\ref{sec:main-result-1}. In order to avoid any confusion, we denote by $\uper$ the homogenized solution associated with the \emph{periodic} Hamiltonian~$\Hper$ in equation~\eqref{eq:8}. We now intend to prove that
\begin{equation}
\label{eq:conv-infini}
\left|u(x)\,-\,\uper(x)\right|\,\to\,0\,\quad\textrm{as}\quad |x|\to +\infty.
\end{equation}
This property of course mathematically encodes that the local defect (that is, the difference between $(f,\ell)$ and~$(\fper,\ellper)$ assumed compactly supported near the origin) only affects the homogenized solution \emph{locally}. 
In Section~\ref{sec:1D}, we have seen that it is  true in the one-dimensional setting and for the separate Hamiltonian considered,  and  even that,  if $\ellper$ is not constant, then $u(x)\,=\,\uper(x)$ outside a neighborhood of the origin.
%This turns out to be  specific to that setting (indeed related to the fact that, then, the homogenized Hamiltonian~$\Hper=|p|$ is flat at $p=0$ and the landscape associated to the cost function $\ell$ is peculiar in a one-dimensional space \textemdash we will comment upon this further below). 

\medskip

The proof of~\eqref{eq:conv-infini} turns out to be rather simple. 
Let us consider, say, the first canonical vector~$\mathbf{e}_1$ in $\R^d$, and the function~$u_n(x)\,=\,u(x\,-\,n\,\mathbf{e}_1)$ obviously solution to the same problem as~\eqref{eq:11} through~\eqref{eq:14}, but with an effective Dirichlet condition set at the point~$x_n\,=\,\,n\,\mathbf{e}_1$ instead of the origin. 
The sequence of functions~$u_n$ inherits, uniformly in~$n\in\N$, of the properties of~$u$ itself. We therefore have uniform almost everywhere pointwise bounds on~$u_n$ and $Du_n$. 
It follows that, up to an extraction (and a diagonal argument) which we henceforth omit to explicitly denote, $u_n$ locally uniformly converges to some function~$v$. 
Fix now a unit ball~$B_1(y_0)$ around an arbitrary point~$y_0\in\R^d$. For any sufficiently large~$n\in\N$, the point~$x_n$ lies outside that ball, and the function $u_n$ thus solves (in the viscosity sense) the periodic homogenized equation~\eqref{eq:11}, that is 
\begin{equation*}
 \alpha u_n+\overline H(Du_n)= 0
\end{equation*}
on that ball.  
Given the above convergence and using the result of stability of viscosity solutions for that equation, so does the limit~$v$. By uniqueness of the bounded uniformly continuous viscosity solution to~\eqref{eq:11} posed on the entire space~$\R^d$, we thus know that~$v=\uper$. 
Since  the limit is unique, we conclude that the sequence~$u_n$ itself converges to $\uper$ uniformly locally. The convergence~\eqref{eq:conv-infini} follows.

\medskip

Two remarks are in order.

\medskip

First, we note that we have only used the homogenized form of the equation provided by Theorem~\eqref{sec:main-result-1} and \emph{not} the assumption \emph{itself} that  $f=\fper$ and~$\ell=\ellper$ outside a neighborhood of the origin. So, if Theorem~\ref{sec:main-result-1} turns out to hold true for some defect $\ell_0$ that does not necessarily have compact support (or alternatively and more generally,  in some specific settings, for  more general  $(f,\ell)$ that converge to $(\fper,\ellper)$ in a milder sense), the convergence~\eqref{eq:conv-infini} still holds true. 
On the other hand and as briefly mentioned above, unless peculiar conditions are met, there is no reason for $u$ and $\uper$ to exactly agree outside a bounded domain. 

\medskip

Second, both above facts may be intuitively understood when resorting to the optimal control interpretation of the equations and solutions at play (in the spirit of our arguments of Section~\ref{sec:function--underline}). 
When starting farther and farther away from the defect located at the origin, the optimal trajectory 
has lesser reason to visit the neighborhood of the origin. As~$|x|\to +\infty$, the value function therefore becomes  decreasingly affected by the presence of the defect. Thus the convergence~\eqref{eq:conv-infini}. 
Additionally and depending upon the specific landscape, such an excursion toward the origin may come, or not, at a price. 
If for points $x$ located far enough from the origin,  this price becomes too high, then the  optimal paths leaving $x$
do not visit the support of the defect. In this case, the exact equality~$u=\uper$ may hold outside a bounded domain. 
%For instance, if $f=\fper$ everywhere, there is an actual (strictly positive) price to pay if exiting the wells of~$\ellper$ is expensive,
% in which case the optimal path does not visit the origin,  and
This is  what happens for a nontrivial one-dimensional setting as that of Section~\ref{ssec:1D-1-periodic}, but not for that of Section~\ref{ssec:1D-zero} or for a higher dimensional setting where some specific paths escaping from the wells of $\ellper$ can be cost-advantageous. 
In the latter situations, only~\eqref{eq:conv-infini} holds.

\medskip

\subsubsection{Randomized variant}
\label{sssec:proba-dD}

\medskip

We devote this section to randomized lattices of points defects in $\R^d$, $d>1$.  We focus our attention to the case when the background environment is constant. The extension to random perturbations of periodic Hamiltonians is left for future work, because  it is not obvious that a counterpart of formula \eqref{eq:random-min-formula}  holds in this situation.

\paragraph{Radially symmetric situations}

We consider the Hamiltonians
\begin{equation}
  H\left(y,p\right)=  \overline H (p)\,-\, \ell_0\left(y\right),
\end{equation}
and for each~${\mathbf k}\in\Z^d$,
%Instead of the Hamiltonian function $\displaystyle H(\frac{x}{\varepsilon},p)$ defined from~\eqref{eq:1}, we consider
\begin{equation}
\label{eq:random-H-2D-k}
H_{{\mathbf k}}^\varepsilon(x\,,\,p\,,\,\omega)\,=
\,\, \overline H (p)\,-\, \ell_0\left(\frac{{x-\varepsilon\,{q}\,{\mathbf k}}}{\varepsilon}\right) X^\eta_{{\mathbf k}}(\omega),
%\sum_{{\mathbf k}\in\Z^d}X^\eta_{\mathbf k}(\omega)\,\ell_0\left(\frac{x-\varepsilon\,{q}\,{\mathbf k}}{\varepsilon}\right),
\end{equation}
where $\displaystyle\left\{X^\eta_{{\mathbf k}}(\omega)\right\}_{{\mathbf k}\in\Z^d}$ is a set of i.i.d. real valued, Bernoulli random variables of parameter~$\eta$, that is, a multi-dimensional analogue of the sequence introduced in Section~\ref{ssec:proba}.  We consider the case when
\begin{itemize}
\item $\overline H$ is a $C^2$, convex, globally Lipschitz and  radially symmetric function of $p$ (it therefore reaches its minimum at $p=0$)
\item $\ell_0$ is smooth, non positive and radially symmetric function,
  supported in the unit  ball centered at the origin, and such that $\argmin \ell_0=\{0\}$ and $\ell_0(0)<0$
  \item  the ergodic  constant $E$ defined in Section~\ref{sec:ergod-const-assoc} satisfies
\begin{equation}
\label{eq:E-inf-H0}
E\,>\,\overline H (0).
\end{equation}
 We will see later that \eqref{eq:E-inf-H0} is in fact a consequence of the assumptions made on $\ell_0$.
\end{itemize}
The Hamiltonian~$H_{{\mathbf k}}^\varepsilon$ corresponds to a rescaled Hamiltonian similar to~$\displaystyle H\left(\frac{x}{\varepsilon},p\right)$, when the local defect (originally set at the origin) is shifted at each of the positions~$\varepsilon\,{q}\,{\mathbf k}$, and only occurs there with the probability~$\eta$, encoded in the random variable~$X_{{\mathbf k}}(\omega)$. The parameter~${q}$ is a positive integer that will be adjusted in the course of the proof, see~\eqref{eq:84} below.
We then form the random Hamiltonian
\begin{equation}
\label{eq:random-H-2D}
H_S^\varepsilon(x\,,\,p\,,\,\omega)\,=\,\sup_{{\mathbf k}\in\Z^d} H_{{\mathbf k}}^\varepsilon(x\,,\,p\,,\,\omega),
\end{equation}
which  also reads as
\begin{equation}
\label{eq:random-H-2D-separate}
H_S^\varepsilon(x\,,\,p\,,\,\omega)\,=\,\,\overline H(p)\,-\,\sum_{{\mathbf k}\in\Z^d}X^\eta_{\mathbf k}(\omega)\,\ell_0\left(\frac{x-\varepsilon\,{q}\,{\mathbf k}}{\varepsilon}\right).
\end{equation}
This is the $d$-dimensional counterpart of the one-dimensional Hamiltonian considered in Section~\ref{ssec:proba}.
The expression~\eqref{eq:random-H-2D-separate} proceeds from the definition~\eqref{eq:random-H-2D-k}-\eqref{eq:random-H-2D} along with the observation that the supports of the two functions~$\displaystyle \ell_0\left(\frac{x-\varepsilon\,{q}\,{\mathbf k}}{\varepsilon}\right)$ and~$\displaystyle \ell_0\left(\frac{x-\varepsilon\,{q}\,{\mathbf k'}}{\varepsilon}\right)$ for different indices~${\mathbf k}\,\not=\,{\mathbf k'}$ do not overlap and the latter two functions are nonpositive, so, for all~$x\in\R^d$ and~$\varepsilon>0$, hence
$$\ell_0\left(\frac{x-\varepsilon\,{q}\,{\mathbf k}}{\varepsilon}\right)\,+\, \ell_0\left(\frac{x-\varepsilon\,{q}\,{\mathbf k'}}{\varepsilon}\right)\,=\,\min\left[\ell_0\left(\frac{x-\varepsilon\,{q}\,{\mathbf k}}{\varepsilon}\right)\,,\, \ell_0\left(\frac{x-\varepsilon\,{q}\,{\mathbf k'}}{\varepsilon}\right)\right].$$

\medskip

We now consider, for almost all~$\omega$, the equation
\begin{equation}
\label{eq:random-HJ-equ}
\alpha\,{U}_\varepsilon(x,\omega)\,+\,H_S^\varepsilon(x\,,\,D{U}_\varepsilon\,,\,\omega)\,=\,0.
\end{equation}
We intend to first prove that this equation has a unique, bounded uniformly continuous solution ${U}_\varepsilon(x,\omega)$ in the viscosity sense, and next study the limit of this solution as~$\varepsilon$ vanishes. For this purpose, under a suitable assumption on $q$, we will in fact explicitly characterize the solution~${U}_\varepsilon$ as
\begin{equation}
\label{eq:random-HJ-U}
{U}_\varepsilon(x,\omega)\,=\,\inf_{{\mathbf k}\in\Z^d}\left\{X^\eta_{\mathbf k}(\omega)\,u_\varepsilon(x-\varepsilon\,{q}\,{\mathbf k})\right\},
\end{equation}
where~$u_\varepsilon(x)$ is the unique viscosity solution in $\rm{BUC}(\R^d)$ to~\eqref{eq:7}, that is
\begin{equation*}
\alpha u_\varepsilon+ H\left(\frac x \varepsilon, D u_\varepsilon\right)=0\quad \hbox{ in }\R^d.
\end{equation*}
The expression~\eqref{eq:random-HJ-U} evidently generalizes~\eqref{eq:random-min-formula} to the present multi-dimensional context.

In order to establish~\eqref{eq:random-HJ-U}, it is obviously sufficient to understand the setting where only two defects are present and are respectively localized, say, at the origin~$x_0=0$  and at $x_{\mathbf k}=\varepsilon\,{q}\,{\mathbf k}$,  for some~${\mathbf k}\in\Z^d\setminus\{0\}$. We thus need to manipulate the two Hamiltonians
$H^\varepsilon_{\mathbf k}$ and $H^\varepsilon_0$ (the latter is obtained by setting ${\mathbf k}=0$ in \eqref{eq:random-H-2D-k}). 
% \begin{equation*}
% \left\{
% \begin{array}{l}
% \displaystyle H^\varepsilon_0\left(x,p,\omega\right)\,=\,\,\overline H(p)\,-\,X^\eta_0(\omega)\,\ell_0\left(\frac{x}{\varepsilon}\right),  \\
%   \\
% \displaystyle H^\varepsilon_{\mathbf k}\left(x,p,\omega\right)\,=\,\,\overline H(p)\,-\,X^\eta_{\mathbf k}(\omega)\,\ell_0\left(\frac{x}{\varepsilon}-\,{q}\,{\mathbf k}\right).
% \end{array}
% \right.
% \end{equation*}
In this simplified setting, we claim that the solution to 
\begin{equation}
\label{eq:random-HJ-equ-simplifie}
\alpha\,{V}_\varepsilon(x,\omega)\,+\,\max\left\{H^\varepsilon_0\,,\,H^\varepsilon_{\mathbf k}\right\}\left(x,D{V}_\varepsilon,\omega\right)
\,=\,0
\end{equation}
actually reads as
\begin{equation}\label{eq:random-HJ-U-simplifie}
  {V}_\varepsilon(x,\omega)\,=\,\min
  \biggl\{ X^\eta_0(\omega)\,u_\varepsilon(x)\,+\,(1-X^\eta_0(\omega))\,\overline u\, ,
                                          X^\eta_{\mathbf k}(\omega)\,u_\varepsilon(x-\varepsilon\,{q}\,{\mathbf k})\,+\,(1-X^\eta_{\mathbf k}(\omega))\,\overline u\biggr\},
                                        \end{equation}
                                        where $u_\varepsilon$  solves~\eqref{eq:7}
                                       % corresponds to  one defect at the origin and, that is
%\begin{equation*}`
%\alpha\,{u}_\varepsilon(x)\,+\,H\left(\frac{x}{\varepsilon},D{u}_\varepsilon\right)
%\,=\,0,
%\end{equation*}
                                        in the sense of viscosity, and the constant function $\overline u\,=-\overline H(0)/\alpha$ corresponds to the situation when there is no defect.
                                        % is the solution of $\alpha \overline u + \overline H (D \overline u)$ and corresponds  without any defect.

If we temporarily admit that~\eqref{eq:random-HJ-U-simplifie} holds true, then the generalization to infinitely many randomized defects is immediate and yields~\eqref{eq:random-HJ-U}, where we note that all the terms in $(1-X^\eta_{{\mathbf k}}(\omega))\overline u$ originally present in~\eqref{eq:random-HJ-U-simplifie} may be discarded because, with full probability, one at least of the $X^\eta_{{\mathbf k}}(\omega)$, ${\mathbf k}\in\Z^d$, has value one and $u_\varepsilon(\cdot-\varepsilon\,{q}\,{\mathbf k})\leq \overline u$ by the comparison principle (using
the non-positiveness of  $\ell_0$).

\medskip

In order to establish~\eqref{eq:random-HJ-U-simplifie}, we proceed as follows. 
First, when~$X^\eta_0(\omega)\,=\,X^\eta_{\mathbf k}(\omega)\,=\,0$,
$H^\varepsilon_0(\cdot,p,\omega)\,=\,H^\varepsilon_{\mathbf k}(\cdot,p,\omega)\,=\,\overline H\left(p\right)$, the constant function  $\overline u$ is the solution to~\eqref{eq:random-HJ-equ-simplifie} and~\eqref{eq:random-HJ-U-simplifie} holds. 
Second, if $X^\eta_0(\omega)\,=\,1$ while~$X^\eta_{\mathbf k}(\omega)\,=\,0$, then~$\displaystyle H^\varepsilon_0 (\cdot, p,\omega)\,=H\left(\frac{\cdot}{\varepsilon},p\right)$ and~$H^\varepsilon_{\mathbf k}  (\cdot, p,\omega)\,=\,\overline H\left(p\right)$, thus, given the non-positiveness of $\ell_0$,
\begin{displaymath}
\max\left\{H^\varepsilon_0  (\cdot, p,\omega) \,,\,H^\varepsilon_{\mathbf k} (\cdot, p,\omega)\right\}\,=\,H\left(\frac{\cdot}{\varepsilon},p\right)  
\end{displaymath}
 and the solution is indeed $u_\varepsilon$. We conclude "symmetrically" when $X^\eta_0(\omega)\,=\,0$ and~$X^\eta_{\mathbf k}(\omega)\,=\,1$. The only interesting case is the third one, when~$X^\eta_0(\omega)\,=\,X^\eta_{\mathbf k}(\omega)\,=\,1$, and we wish to prove that
\begin{equation}
\label{eq:random-HJ-V-simplifie}
{V}_\varepsilon(x,\omega)\,=\,\min\biggl\{\,u_\varepsilon(x)\,,\,u_\varepsilon(x-\varepsilon\,{q}\,{\mathbf k})\biggr\}.
\end{equation}
To address the latter case, we partition $\R^d$ into three non-overlapping regions, namely ${\mathcal A}_0^\varepsilon\,=\,\displaystyle \textrm{Supp}\;\ell_0\left(\frac{\cdot}{\varepsilon}\right)$, $\displaystyle {\mathcal A}_{\mathbf k}^\varepsilon\,=\,\textrm{Supp}\;\ell_0\left(\frac{\cdot}{\varepsilon}\,-\,{q}\,{\mathbf k}\right)\,=\,\varepsilon\,{q}\,{\mathbf k}\,+\,\textrm{Supp}\;\ell_0\left(\frac{\cdot}{\varepsilon}\right)$, and the remaining part $\displaystyle \left({\mathcal A}_0^\varepsilon\,\cup\, {\mathcal A}_{\mathbf k}^\varepsilon\right)^c$  of~$\R^d$.

 From the non-positiveness of $\ell_0$, we notice that
 \begin{equation}
 \label{eq:0-1-01c}
\max\left\{H^\varepsilon_0\,,\,H^\varepsilon_{\mathbf k}\right\}\,=\,
\left\{
\begin{array}{ll}
H^\varepsilon_0&\textrm{in}\;{\mathcal A}_0^\varepsilon\,,\\
H^\varepsilon_{\mathbf k}&\textrm{in}\;{\mathcal A}_{\mathbf k}^\varepsilon\,,\\
\displaystyle H^\varepsilon_0\,=\,H^\varepsilon_{\mathbf k}\,=\,\overline H&\textrm{in}\;\left({\mathcal A}_0^\varepsilon\,\cup\, {\mathcal A}_{\mathbf k}^\varepsilon\right)^c\,.
\end{array}
\right.
\end{equation}
We are going to prove that 
% Now,
% , possibly under further assumptions and
for $\varepsilon$ sufficiently small, 
 \begin{equation}
  \label{eq:0-1-sol}
\left\{
\begin{array}{ll}
u_\varepsilon(x)\,\leq\,u_\varepsilon(x-\varepsilon\,{q}\,{\mathbf k})&\textrm{in}\;{\mathcal A}_0^\varepsilon\,,\\
u_\varepsilon(x-\varepsilon\,{q}\,{\mathbf k})\,\leq\,u_\varepsilon(x)&\textrm{in}\;{\mathcal A}_{\mathbf k}^\varepsilon\,.\\
\end{array}
\right.
\end{equation}
%for all~${\mathbf k}\in Z^d$.
 Assume temporarily that \eqref{eq:0-1-sol} is true.
Since~$u_\varepsilon$ and $u_\varepsilon(\cdot-\varepsilon\,{q}\,{\mathbf k})$ are respectively viscosity solutions to~\eqref{eq:7} and that same equation translated from~$-\varepsilon\,{q}\,{\mathbf k}$, and  since they are both locally Lipschitz continuous, we know from~\cite[Prop. 1.9, Chapter I]{MR1484411} that they are solutions almost everywhere of those equations, respectively. 
 It follows from~\eqref{eq:0-1-01c} and~\eqref{eq:0-1-sol}  together, that~$\displaystyle \min\bigl\{\,u_\varepsilon(x)\,,\,u_\varepsilon(x-\varepsilon\,{q}\,{\mathbf k})\bigr\}$ is solution almost everywhere to~\eqref{eq:random-HJ-equ-simplifie}. 
 We also know that this minimum is a viscosity supersolution of the equation. Finally,  since the minimum of two Lipschitz continuous functions is also Lipschitz continuous and since the Hamiltonian is convex, we know from~\cite[Prop. 5.1, Chapter II]{MR1484411} that this almost everywhere, Lipschitz continuous, solution is also a viscosity subsolution. We thus conclude  that it is a viscosity solution and that~\eqref{eq:random-HJ-V-simplifie}. Consequently,  if~\eqref{eq:0-1-sol} holds for all~${\mathbf k}\in \Z^d \setminus\{0\}$, then~\eqref{eq:random-HJ-U} holds.
 
\medskip

The inequations in \eqref{eq:0-1-sol} in fact only rely upon an interplay between the properties of the solution $u_\varepsilon$ for the problem with \emph{one} defect (essentially that $u_\varepsilon(x)$ grows as $x$ departs from the defect) and the distance between the \emph{two contiguous} defects considered. At this point, we may assume  that $\alpha=1$, just for alleviating notation.

To start with, we observe that the  solution $u$ to the Dirichlet problem~\eqref{eq:11}-\eqref{eq:14} in the sense of Theorem~\ref{sec:main-result-1}, (a) is radially symmetric and (b) is a strictly increasing function of the radius $|x|$. The radial symmetry comes from the uniqueness of that solution stated in Theorem~\ref{sec:main-result-1}. The increasing character comes from our assumption~\eqref{eq:E-inf-H0}. 
Indeed, because of this assumption and~\eqref{eq:12}, we have 
$u(0)\,\leq\, -E\,<\, -{\overline H}(0)$. 
By  continuity,  
$ u(x)\,<\, -{\overline H}(0)$  on a neighborhood $B_R(0)$ for some radius~$R>0$. 
The continuity of $\overline H$ then implies that $|D u(x)|\geq \delta >0$ for some $\delta >0$ and almost all $x\in B_R(0)$. By radial symmetry, this amounts to~$\displaystyle\left|\frac{\partial u}{\partial r}\right|\geq \delta >0$ for almost all~$r\leq R$. But the notion of viscosity solution then implies that, on $B_R(0)$, $\displaystyle\frac{\partial u}{\partial r}$, which is strictly positive around the origin (given that we have already established that $u$ reaches  its global infimum there) can change sign and become negative at most once. If this is the case, then $u$ decreases on the region where ~$\displaystyle\frac{\partial u}{\partial r}\leq\,- \delta$, so the inequality 
$u(x)\,<\, -{\overline H}(0)$ is all the more true. We may thus enlarge the ball $B_R(0)$ and since $|D u|$ cannot be arbitrarily small on this region, the viscosity solution~$u$ cannot become increasing again. Therefore, $u$ remains decreasing as~$|x|$ grows to infinity. On the other hand, we also already know from~\eqref{eq:conv-infini} that $
\left|u(x)\,-\,\overline u\right|\,\to\,0$ at infinity and from the non-positiveness of~$\ell_0$ that $u\leq \overline u$, so we reach a contradiction. In conclusion,  $\displaystyle\frac{\partial u}{\partial r}\geq \delta >0$ everywhere outside the origin and $u$ is indeed a strictly increasing, radially symmetric function.

\medskip

We now return to $u_\varepsilon$, for which we will apply a similar argument, using the uniformity of properties with respect to $\varepsilon$. To start with, we notice that, because $\ell_0$ is assumed radially symmetric and $\overline H$ is isotropic, $u_\varepsilon$  is also radially symmetric.   Next,  since $\ell_0$ is nonpositive, comparison  yields as above that $u_\varepsilon\,\leq\, \overline u \,=\,-\,\overline{H}(0)$ in ~$\R^d$. On the other hand, by Theorem~\ref{sec:main-result-1}, $u_\varepsilon(0)\to u(0)\,\leq\,-E<-\overline{H}(0)$.

Let us fix an arbitrary positive constant $\beta$ such that $0\,<\,2 \beta\,< \,E-\,\overline{H}(0)$. There exists a positive radius $R_0$ such that
$u (x) < -\overline H (0) -2 \beta$ on $\overline {B_{R_0} (0)}$. From the uniform convergence of  $u_{\varepsilon}$ to $ u$ on  $\overline {B_{R_0} (0)}$, we know that there exists $\varepsilon_0>0$ such that for any $\varepsilon$ such that $0< \varepsilon\le \varepsilon_0$,
\begin{equation}\label{eq:77}
 u_\varepsilon (x) < -\overline H (0) - \beta, \quad  \hbox{ on }\overline {B_{R_0} (0)}.
\end{equation}

\paragraph{Step 1} Let us start by studying $u_\epsilon$ in  $B_{R_0}(0)\setminus B_\varepsilon(0)$.
Given that $H_\varepsilon$ coincides with
$\overline H$ outside  $\displaystyle \textrm{Supp}\;\ell_0\left(\frac{\cdot}{\varepsilon}\right)$, \eqref{eq:77} implies that there exists some $\delta >0$, independent of~$\varepsilon\leq\varepsilon_0$, such that
\begin{equation}
\label{eq:borne-en-dessous-Dueps-prebis}
|Du_\varepsilon(x)|\,\geq \,\delta \,>\,0,\quad\textrm{for almost all}\;\varepsilon\,\leq |x|\,\leq\,R_0.
\end{equation}
Therefore, for almost all $\varepsilon\,\leq |x|\,\leq\,R_0$,
\begin{equation}
\label{eq:modulebis}
\displaystyle \left|\frac{\partial u_\varepsilon}{\partial r}\right|\geq \,\delta.
\end{equation}
 Since we know,  from the notion of viscosity solution and from~\eqref{eq:borne-en-dessous-Dueps-prebis}, that $Du_\varepsilon$ cannot be arbitrarily small on  $B_{R_0}(0)\setminus B_\varepsilon(0)$,  $\displaystyle \frac{\partial u_\varepsilon}{\partial r}$ may change sign at most once in this ring, and if it is the case, then the jump is from positive to negative as $|x|$ grows.

 Hence, if $\lim_{r\to \varepsilon_-} \displaystyle \frac{\partial u_\varepsilon}{\partial r} (r) <0$, then   $\displaystyle \frac{\partial u_\varepsilon}{\partial r}\le -\delta$ in  $B_{R_0}(0)\setminus B_\varepsilon(0)$.  This implies that we may find a larger $R_0$ still satisfying $u_\varepsilon(x)\,\leq -\overline{H}(0) -\beta$ for $x \in B_{R_0}(0)\setminus B_\varepsilon(0) $.  Repeating the argument, we obtain that we may indeed choose $R_0=+\infty$ and that $\lim _{|x|\to \infty} u_\varepsilon(x)=-\infty$, which contradicts the convergence of  $u_\varepsilon$  to $-\overline H(0)$ at infinity.

Therefore, $ \lim_{r\to \varepsilon_-} \displaystyle \frac{\partial u_\varepsilon}{\partial r} (r) >0$ and $\displaystyle \frac{\partial u_\varepsilon}{\partial r}$ may change sign at most once on \\ $B_{R_0}(0) \setminus {B_\varepsilon(0)} $ to become negative, say at some radius~$R_1\,<\, R_0$, where (this is therefore the only possible case) it jumps: \begin{equation}
\label{eq:saut}
\left\{
\begin{array}{ll}
\displaystyle \frac{\partial u_\varepsilon}{\partial r}\,\geq\,\delta &\textrm{for}\;\varepsilon\leq |x|< R_1,\\
\displaystyle \frac{\partial u_\varepsilon}{\partial r}\,\leq\,-\,\delta &\textrm{for}\;R_1< |x|\leq R_0.
\end{array}
\right.
\end{equation}
We claim that this situation cannot occur. Indeed, if it was the case, then $u_\varepsilon(x)\,< -\overline{H}(0) -\beta$ for $|x|=R_0$ and we would be able to  enlarge $R_0$ while keeping the inequality true. Then  $u_\varepsilon$ would stay decreasing in the whole region~$|x|\geq R_1$, which again would contradict  the convergence of $u_\varepsilon$ to $-\overline H (0)$ at infinity.

To summarize, we have proven that
\begin{equation}
  \label{eq:78}
\displaystyle \frac{\partial u_\varepsilon}{\partial r}\ge \delta, \quad  \hbox{ in  the ring } \{ x: \varepsilon \le |x| \le R_0\}.
\end{equation}

\paragraph{Step 2} Let us now study $u_\varepsilon$  in $B_\varepsilon(0)$.

Assume first that $u_\varepsilon$ has a  local maximum at $0$.
Then,  $u_{\varepsilon}(0)\le -\overline H (0) +\ell_0(0)$, and the function  $u_{\varepsilon}$ cannot have a  minimum at some $\underline x$, $0<|\underline x|<\varepsilon$. Indeed, if it was the case,  $u_{\varepsilon}(\underline x)\ge -\overline H (0) +\ell_0(\underline x)> -\overline H (0) +\ell_0(0)\ge u_{\varepsilon}(0) $, which is contradictory. Therefore, the minimum of  $u_\varepsilon$ in $\overline {B_{\varepsilon}(0)}$ is achieved at $|x|=\varepsilon$, which is impossible because 
 $ \lim_{r\to \varepsilon_-} \displaystyle \frac{\partial u_\varepsilon}{\partial r} (r) >0$. We have proved by contradiction that $u_\varepsilon$ does not have  a  local maximum at $0$.

%  and its value is not larger than   $ -\overline H (0) +\ell_0(0)$. Hence $ \frac{\partial u_\varepsilon}{\partial r}$ is not positive near $|x|=\varepsilon$. Combining this with \eqref{eq:modulebis}, we see that $ \frac{\partial u_\varepsilon}{\partial r}\le -\delta$ near $|x|=\varepsilon$.  Then, from the observation made above, $ \frac{\partial u_\varepsilon}{\partial r}$ must not excede  $-\delta$ in the whole ring
% $\{ x: \varepsilon\le |x| \le R_0\}$. This implies that we may may find a larger $R_0$ still satisfying $u_\varepsilon(x)\,\leq -\overline{H}(0) -\beta$ for $x \in B_{R_0}(0)\setminus B_\varepsilon(0) $.  Repeating the argument, we obtain that we may indeed choose $R_0=+\infty$ and that $\lim _{|x|\to \infty$ u_\varepsilon(x)=-\infty$, which contradicts the convergence of  $u_\varepsilon$  to $\overline u$ at infinity. We have proved by contradiction that $u_\varepsilon$ does not reach a  local maximum at $0$.

On the other hand, the semi-concavity and the radial symmetry of  $u_\varepsilon$ imply that $\frac {\partial u_\varepsilon} {\partial r}$ has a limit at $0$, which is nonpositive.

Combining the latter two points yields  that $u_\varepsilon $ has a derivative at $0$ and that
  $\frac {\partial u_\varepsilon} {\partial r}(0)=0$. This implies that
  \begin{equation}
    \label{eq:79}
u_\varepsilon(0)= -\overline H (0) +\ell_0(0),    
  \end{equation}
  and from the convergence of $u_\varepsilon(0)$ to $u(0)$, that
  \begin{equation}
    \label{eq:80}
E= \overline H (0) -\ell_0(0).       
\end{equation}
We have also proved that $-E$ is the minimal value of $u_\varepsilon$ on $B_\varepsilon(0)$. Note that  \eqref{eq:79} and \eqref{eq:80} may be obtained in an easier way by using arguments from the theory of optimal control.

\paragraph{Step 3} Our next step consists of proving that it possible to choose a  positive integer $\underline q$ such that 
\begin{equation}
  \label{eq:81}
  \min_{|x|\ge \underline q \varepsilon } u_{\varepsilon}(x) \ge \max _{|x|\le  \varepsilon } u_{\varepsilon}(x). 
\end{equation}
First, we know that there exists $M>0$ independent of $\varepsilon$ such that $\|D u_\varepsilon\|_\infty\le M$. This and  \eqref{eq:79}-\eqref{eq:80} imply that 
\begin{equation}
   \label{eq:82}
   \max _{|x|\le  \varepsilon } u_{\varepsilon}(x)\le -E + M \varepsilon.
\end{equation}
On the other hand, \eqref{eq:78} implies
\begin{equation}
  \label{eq:83}
  \begin{array}[c]{rcl}
    \ds   \min _{     \underline q\varepsilon \le |x| \le R_0} u_{\varepsilon}(x) &\ge&  u_{\varepsilon}(|y|=\varepsilon) + \delta (\underline q-1) \epsilon     \\
     &\ge&  -E + \delta (\underline q-1) \epsilon .
  \end{array}
\end{equation}
We deduce from  \eqref{eq:82} and  \eqref{eq:83} that if $\underline q>1+ M/\delta$, then $ \min_{q \varepsilon\le |x| \le R_0 } u_{\varepsilon}(x) \ge \max _{|x|\le  \varepsilon } u_{\varepsilon}(x)$.

Moreover,  $ u_{\varepsilon}$ cannot reach a value smaller than $ \max _{|x|\le  \varepsilon } u_{\varepsilon}(x)$ at some $y$ such that $|y|>R_0$, because, if it was the case, then the infimum of $u_{\varepsilon}|_{\{|x|\ge R_0\}}$ would be 
achieved by some  $\underline x$, $R_0< |\underline x|$, and 
\begin{displaymath}
 \max _{|x|\le  \varepsilon } u_{\varepsilon}(x) \ge   u_\epsilon(\underline x)\ge -\bar H(0),
\end{displaymath}
which is a contradiction for $\varepsilon\le \varepsilon_0$. Hence  \eqref{eq:81} holds for   $\underline q>1+ M/\delta$ and $\varepsilon \le \varepsilon_0$.

It is then easy to obtain   \eqref{eq:0-1-sol} for any  $\varepsilon \le \varepsilon_0$, with
\begin{equation}
  \label{eq:84}
q= \underline q+2, \quad \underline q>1+ M/\delta.
\end{equation}

Finally, the explicit expression~\eqref{eq:random-HJ-U} of the solution to~\eqref{eq:random-HJ-equ} is established, it is a straightforward adaptation of our results of Section~\ref{ssec:proba} to indeed identify the limit  as $\varepsilon$ vanishes.  All our arguments and conclusions hold \emph{mutatis mutandis} with minor modifications.

\paragraph{Generalization}

The results obtained above can be  generalized to a class of non-radially symmetric situations under the following assumptions:

\begin{itemize}
\item
  The Hamiltonian $\overline H$ defined by  $\overline H (p) = \max_{a\in A} -\overline f(a)\cdot p - \overline \ell(a)$ is a   $C^ 2$, convex, globally Lipschitz   function which reaches its minimum at $p=0$. This implies that $\overline H (0) =- \min_{a\in A}\overline \ell(a)= - \min_{a\in A: \overline f(a)=0} \overline \ell(a)$
\item The function $\ell_0: \R^d\times A \to \R_- $ is smooth, and for all $a\in A$, $\ell_0(\cdot, a)$ is supported in the unit ball $B_1(0)\subset \R^d$.
  The perturbed Hamiltonian is $H(y,p)=  \max_{a\in A} \overline f(a)\cdot p - \overline \ell(a) -\ell_0(y,a)$
\item  The ergodic constant $E$ defined in 
Section~\ref{sec:ergod-const-assoc} satisfies \eqref{eq:E-inf-H0}
\item  There exists two  Hamiltonians $\overline H_1$ and  $\overline H_2$
  which are radially symmetric, $C^ 2$, convex, globally Lipschitz fonctions defined on $\R^d$ and
 two smooth, radially symmetric functions  $\ell_{0,1}$ and $\ell_{0,2}$,
 defined on $\R^d$ with values in  $\R_-$, supported in $B_1(0)$, 
  such that the following conditions are satisfied:
 \begin{itemize}
 \item   the pairs $(\overline H_1, \ell_{0,1})$ and  $(\overline H_2, \ell_{0,2})$ satisfy all the assumptions made in the radially symmetric setting.  In 
particular, as we have seen above, the related ergodic constants $E_1$ and $E_2$ 
are given by $E_1= \overline H_1(0) -  \ell_{0,1}(0)$ and   $E_2= \overline H_2(0) -  \ell_{0,2}(0)$
\item  there holds
  $\overline H_1(p)\le \overline H(p) \le \overline H_2(p)$ for all $p\in \R^d$,
  $\ell_{0,1}(x)\ge \ell_0(x,a) \ge \ell_{0,2}(x)$  for all $x\in \R^d$, $a\in A$.
  Moreover, $\overline H_1(0)= \overline H(0) =\overline H_2(0)$ and  $\ell_{0,1}(0)=\ell_{0,2}(0)$. 
 \end{itemize}
  The corresponding perturbed Hamiltonians are $H_1(y,p)= \overline H_1(p) -\ell_{0,1}(y)$ and  $H_2(y,p)= \overline H_2(p) -\ell_{0,2}(y)$. Obviously,
  \begin{equation}
    \label{eq:85}
   H_1(y,p)\le H(y,p)\le H_2(y,p),    
 \end{equation}
 and it can be checked that the assumptions imply that $E_1= E= E_2$.
\end{itemize}
Using  \eqref{eq:85} and comparison principles, it is possible to prove that
\eqref{eq:random-HJ-U-simplifie} still holds in this situation, provided that $q$ is chosen large enough, and all the results proved in the radially symmetric setting remain valid. Note that the assumptions made are rather strong: in particular they require that $a \mapsto \ell_0(0,a) $ be constant.

\bibliographystyle{siam}
\bibliography{Hom_HJB_defect}

\end{document}